\documentclass[11pt]{article}
\usepackage{amsmath,amssymb,amsthm,amsfonts,amstext,amsbsy,amscd}
\usepackage{graphics}
\usepackage{graphicx}

\RequirePackage[colorlinks,citecolor=blue,urlcolor=blue]{hyperref}

\RequirePackage{hypernat}

\def\<{\langle}
\def\>{\rangle}

\def\Chi{\raise .3ex
\hbox{\large $\chi$}}

\def\({\Bigl (}
\def\){\Bigr )}
\newcommand{\be}{\begin{equation}}
\newcommand{\ee}{\end{equation}}
\newcommand{\bea}{$$ \begin{array}{lll}}
\newcommand{\eea}{\end{array} $$}
\newcommand{\bi}{\begin{itemize}}
\newcommand{\ei}{\end{itemize}}

\numberwithin{equation}{section}
\newtheorem{satz}{Satz}[section]

\newtheorem{theorem}[satz]{Theorem}

\newtheorem{lemma}[satz]{Lemma}
\newtheorem{assumption}[satz]{Assumption}

\DeclareMathOperator{\E}{{\mathbb E}}

\DeclareMathOperator{\G}{{\mathbb G}}
\DeclareMathOperator{\bH}{{\mathbb H}}

\DeclareMathOperator{\R}{{\mathbb R}}

\DeclareMathOperator{\N}{{\mathbb N}}

\DeclareMathOperator{\PP}{{\mathbb P}}

\renewcommand{\phi}{\varphi}
\renewcommand{\cdot}{{\scriptstyle \bullet} }
\providecommand{\abs}[1]{\lvert #1 \rvert}

\renewcommand{\le}{\leqslant}
\renewcommand{\ge}{\geqslant}

\newcommand{\footnoteremember}[2]{
  \footnote{#2}
  \newcounter{#1}
  \setcounter{#1}{\value{footnote}}
}
\newcommand{\footnoterecall}[1]{
  \footnotemark[\value{#1}]
}

\let\epsilon=\varepsilon
\let\ep=\epsilon
\let\phi=\varphi
 \let\be=\beta  
\let\de=\delta

\let\tilde=\widetilde

\newcommand{\Sph}{\mathbb{S}^2}
\newcommand{\SO}{\mathcal{SO}(3)}

\begin{document}
\title{Blockwise SVD with error in the operator and application to blind deconvolution}
\author{S. Delattre\footnoteremember{adr}{Universit\'e Denis Diderot Paris 7 and CNRS-UMR 7099, 175 rue du Chevaleret 75013 Paris, France. E-mail: sylvain.delattre@univ-paris-diderot.fr, dominique.picard@univ-paris-diderot.fr, thomas.vareschi@univ-paris-diderot.fr} M. Hoffmann\footnote{ENSAE and CNRS-UMR 8050, 3 avenue Pierre Larousse, 92245 Malakoff Cedex, France. E-mail: marc.hoffmann@ensae.fr}\,\,\,\,\,\,D. Picard\footnoterecall{adr} and T. Vareschi\footnoterecall{adr}}

\date{}
\maketitle

\begin{abstract}
We consider linear inverse problems in a nonparametric statistical framework. Both the signal and the operator are unknown and subject to error measurements. We establish minimax rates of convergence under squared error loss when the operator admits a blockwise singular value decomposition (blockwise SVD) and the smoothness of the signal is measured in a Sobolev sense. We construct  a nonlinear procedure adapting simultaneously to the unknown smoothness of both the signal and the operator and  achieving the optimal rate of convergence to within logarithmic terms. When the noise level in the operator is dominant, by taking full advantage of the blockwise SVD property, we demonstrate that the block SVD procedure overperforms classical methods based on Galerkin projection \cite{EK} or nonlinear wavelet thresholding \cite{HR}. We subsequently apply our abstract framework to the specific case of blind deconvolution on the torus and on the sphere.
\end{abstract}

\section{Introduction}
\subsection{Motivation} \label{motivation}
Consider the following idealised statistical problem: estimate a function $f$ (a signal, an image) from data
\begin{equation} \label{Gaussian deconvolution}
y_n = K f  + n^{-1/2} \dot W,
\end{equation}
where 
$$K: \bH \rightarrow \G$$
is a linear operator between two Hilbert spaces $\bH$ and $\G$. 
The observation of the unknown $f \in \bH$ is challenged by the action of the linear degradation $K$ as well as contaminated by an experimental  Gaussian white noise $\dot W$ on $\G$
 with vanishing noise level $n^{-1/2}$ as $n\rightarrow \infty$.
Alternatively, in a density estimation setting, we observe a random sample 
$(Z_1,\ldots, Z_n)$ 
drawn from a probability distribution\footnote{In that setting, $Kf$ must therefore also be a probability density 
.} 
with density $Kf$. 
In each case, we do not know the operator $K$ exactly, but we have access to 
\begin{equation} \label{Kdelta}
K_{\delta} = K+\delta \dot B,
\end{equation}
where $\dot B$ is a Gaussian white noise on $\bH\times\G$
thanks to preliminary experiments or calibration through trial functions. 
This setting has been discussed in details in \cite{EK, HR}. In this paper, we are interested in operators $K$ admitting a singular value decomposition (SVD) or a blockwise SVD. In essence, we know the typical eigenfunctions of $K$ but not the eigenvalues. We cover two specific examples of interest: spherical and circular deconvolution. 
\begin{proof}[Spherical deconvolution] Used for the analysis of data distributed on the celestial sphere, see Section \ref{exemple deconvolution spherique} below.
One observes a random sample $(Z_1,\ldots, Z_n)$ with
$$Z_i = \varepsilon_i X_i,\;\;i=1,\ldots, n$$
where the $\varepsilon_i$ are random elements in $\SO$, the group of $3 \times 3$ rotation matrices,
and the $X_i$ are independent and identically distributed on the sphere $ \Sph$, with common density $f$ with respect to  the uniform probability distribution $\mu$ on $\Sph$. In this setting, if the  $\varepsilon_i$ have common density $g$ with respect to the Haar measure $du$ on $\SO$, we have
$$Kf(x) =g\star f(x) = \int_{\SO}g(u)f(u^{-1}x)du,\;\;x\in \Sph.$$
We are interested in the case where the exact form  $g$ is unknown. However, $K$ is block-diagonal in the spherical harmonic basis.
\end{proof}
\begin{proof}[Circular deconvolution] Used for restoring signal or images, see Section \ref{exemple deconvolution circulaire} below. We take $\bH=\G=L^2({\mathbb T})$ the space
of square integrable functions on
 the torus ${\mathbb T}= [0,1]$ (or $[0,1]^d$) appended with periodic boundary conditions. We have
$$Kf(x)=g \star f(x) = \int_{{\mathbb T}} g(u)f(x-u)du,\;\;x\in {\mathbb T}.$$
The degradation process $K = g \star \cdot$ is characterised by the impulse response function $g$ which we do not know exactly. However, $K$ is diagonal in the Fourier basis. 
\end{proof}
Although the problem of estimating $f$ is fairly classical and well understood when $K$ is known (a selected literature is \cite{Tsyb,CavaTsyb, Donoho, AS, CHR, JKPR, KKLPPP} and the references therein), only moderate attention has been paid in the case of an unknown $K$ despite its relevance in practice. When the eigenfunctions of $K$ are known solely, we have the results of Cavalier and Hengartner \cite{CH}, Cavalier and Raimondo \cite{CavaR} but they are confined to the case  where the error in the operator is negligible $\delta \ll n^{-1/2}$. In a general setting with error in the operator, Efromovitch and Kolchinskii \cite{EK} and later Hoffmann and Rei\ss\, \cite{HR} studied the recovery of $f$ when the eigenfunctions and eigenvalues of $K$ are unknown. In both contributions, a marginal attention is paid to the case of sparse or diagonal operators, but it is showed in both papers that unusual rates of convergence can be obtained when $n^{-1/2} \ll \delta$. In a univariate setting, Neumann \cite{Neumann} and Comte and Lacour \cite{ComteLacour} consider the case of deconvolution with an error density, known only through an auxiliary set of $m$ learning data. This formally corresponds to having $\delta = m^{-1/2}$ in our setting. Minimax rates and adaptive estimators are derived in both regimes $m \ll n$ and $n \ll m$. 
We address in the paper the following program:
\begin{itemize}
\item[i)] Construction of a feasible procedure $\widehat f_{n,\delta}$  estimating $f$ from data \eqref{Gaussian deconvolution} and \eqref{Kdelta} that achieves optimal rates of convergence (up to inessential logarithmic terms). We require $\widehat f_{n,\delta}$ to be adaptive with respect to smoothness constraints on $f$ and $K$.
\item [ii)] Identification of best achievable accuracy for $f$ under smoothness constraints on $f$ and $K$ so that the interplay between $n^{-1/2}$ and $\delta$ can be explicitly related in the asymptotic $\delta \rightarrow 0$ and $n \rightarrow \infty$; this includes the comparison with earlier results of \cite{Neumann, EK, HR} in the context of blockwise SVD.
\item [iii)] Application to spherical deconvolution on $\mathbb{S}^2$ or circular deconvolution on the torus; this includes the discussion of our findings in terms of the existing literature on the topic \cite{CR, KPP, KK} and some practical aspects of numerical implementation.
\end{itemize}
\subsection{Main results and organisation of the paper}
In Section \ref{Estimation by blockwise SVD}, we present an abstract framework that allows for operators $K$ to admit a so-called blockwise SVD. This property is simply turned into the existence of pairs of increasing finite dimensional spaces ($H_\ell,\; G_\ell)$ that are stable under the action of $K$. 
The blockwise SVD property is further appended with a smoothness condition quantified by the arithmetic decay of the operator norm of $K$ and its inverse on $H_\ell$ (resp. $G_\ell$) (the so-called {\it ordinary smooth} assumption, see {\it e.g.} \cite{Tsyb}). By means of a reconstruction formula, we obtain in Section \ref{def estimator} an estimator $\widehat f_{n,\delta}$ of $f$ by first inverting $K_\delta$ on $H_\ell$ with a thresholding tuned with $\delta$ and then filter the resulting signal by a block thresholding tuned with $n^{-1/2}$. As for i) and ii), we establish in Theorems \ref{upper bounds} and \ref{borne inf} of Section \ref{main results} the minimax rates of convergence for Sobolev constraints on $f$ under squared error loss and we demonstrate that $\widehat f_{n,\delta}$ is optimal and adaptive to within logarithmic terms. The explicit interplay between $\delta$ and $n^{-1/2}$ is revealed and discussed in the case of sparse operator when $n^{-1/2} \ll \delta$, completing earlier findings in \cite{EK, HR} and to some extent \cite{Neumann} in the univariate case for density deconvolution. In particular, we demonstrate that a certain parametric regime dominates when the smoothness of the signal dominates the smoothing properties of the operator. Concerning iii), the method is applied to the case of spherical and circular deconvolution in Section \ref{examples} where harmonic Fourier analysis enables to provide explicit blockwise SVD for the convolution operator. We illustrate the numerical feasability of $\widehat f_{n,\delta}$ and the phenomena that appear in the case $n^{-1/2} \ll \delta$. Section \ref{proofs} is devoted to the proofs.

We choose to state and prove our results in the white Gaussian model generated by the observation of $y_n$ and $K_\delta$ defined by \eqref{Gaussian deconvolution} and \eqref{Kdelta}. The extension to the case of density estimation, when $y_n$ is replaced by the observation of a random sample of size $n$ drawn from the distribution $Kf$, like for instance in \cite{Neumann, ComteLacour} is a bit more involved, due to the intrinsic heteroscedasticity that appears when enforcing a formal analogy with the Gaussian setting \eqref{Gaussian deconvolution}. It is briefly addressed in the discussion Section \ref{discussion}
\section{Estimation by blockwise SVD} \label{Estimation by blockwise SVD}
\subsection{The blockwise SVD property} \label{blockwise SVD property}
Let ${\mathcal G}$ denote a family of linear operators
$$K :  \bH \rightarrow  \G $$
between two Hilbert spaces $\bH$ and $\G$ that shall represent our parameter set of unknown $K$.
 
A fundamental property (Assumption \ref{blockwise SVD} below) is that an explicit singular value decomposition (SVD) or blockwise SVD is known for all $K \in {\mathcal G}$ simultaneously. More specifically, we suppose that there exist two explicitly known bases
$(e_\lambda, \; \lambda\; \in \Lambda)$ of $\bH$ and $(g_\lambda, \; \lambda\; \in \Lambda)$ of $\G$, as well as a partition of $\Lambda=\cup_{\ell\geq 1}\Lambda_\ell$ with $\Lambda_\ell\cap\Lambda_{\ell'}=\emptyset$ if $\ell\not=\ell'$, and a constant $d \geq 1$ such that:
$$\ell^{d-1} \lesssim |\Lambda_\ell| \lesssim \ell^{d-1},$$
where $\lesssim$ means inequality up to a multiplicative constant that does not depend on $\ell$.  Here $|\Lambda_\ell| = \text{Card}(\Lambda_\ell)$. 

It is worthwhile to notice that in our examples as well as in the rates of convergence that we will exhibit later, $d$ plays the role of a dimension. In particular, $d=1$ corresponds to a 'standard SVD', whereas $d>1$  creates blocks and deserves the name of 'blockwise' SVD.
However, there is no need in the paper to assume that $d$ is in $\N$. Set 
$$H_\ell=\mathrm{Span}\{e_\lambda,\lambda \in \Lambda_\ell\}\;\;\text{and}\;\;G_\ell=\mathrm{Span}\{g_\lambda, \lambda \in \Lambda_\ell\}.$$
The Galerkin projection of an operator $T:\bH \rightarrow \G$ onto $(H_\ell,G_\ell)$ is defined by $T_\ell=P_\ell T_{|H_\ell}$, where $P_\ell $ is the orthogonal projector onto $G_\ell$. 
\begin{assumption}[Blockwise SVD] \label{blockwise SVD}
$$
{K}_{|H_\ell} = K_\ell\;\;\;\text{for every}\;\;K\in {\mathcal G},\ell\geq 1.$$
\end{assumption}


%

We further need to quantify the action of $K$ on $H_\ell$. We denote by $\|T_\ell\|_{H_\ell \rightarrow G_\ell}=\sup_{v \in H_\ell,\|v\|_{\bH}=1}\|T_\ell v\|_{\G}$ the operator norm of $T_\ell$.
\begin{assumption}[Spectral behaviour of ${K}_{|H_\ell}$] 
\label{spectral behaviour of K}
For every $\ell \geq 1$, $K_\ell$ is invertible and there exists $\nu \geq 0$ such that 
$$
Q_1(K)=\sup_{\ell \geq 1}\ell^{-\nu}\|(K_\ell)^{-1}\|_{G_\ell \rightarrow H_\ell} <\infty$$
and
$$Q_2(K)=\sup_{\ell \geq 1}\ell^{\nu}\|K_\ell\|_{H_\ell \rightarrow G_\ell} <\infty$$
for every $K \in {\mathcal G}$.
\end{assumption}
We associate with the bases $(e_\lambda, \lambda \in \Lambda)$ and $(g_\lambda, \lambda \in \Lambda)$ the following decompositions
$$f = \sum_{\ell \geq 1} \sum_{\lambda \in \Lambda_\ell} \langle f, e_\lambda \rangle\, e_\lambda,\;\; g= \sum_{\ell \geq 1} \sum_{\lambda \in \Lambda_\ell} \langle g, g_\lambda \rangle\, g_\lambda
\;\;\text{
for every}\;\; f \in\bH, \; g \in \G,$$
where $\langle \cdot,\cdot\rangle$ denotes the inner product either in $\bH$ or $\G$  and the scale of Sobolev spaces 
\begin{align} \label{Sobolev}
\mathcal{W}^s &= \Big\{f\in \bH,\;\;\|f\|_{\mathcal{W}^s}^2=\sum_{\ell \geq 1}\ell^{2s}\sum_{\lambda \in \Lambda_\ell} \langle f,e_\lambda\rangle^2 <\infty\Big\},\;\;s\in \R,
\\
\mathcal{\tilde W}^s &= \Big\{g\in \G,\;\;\|g\|_{\mathcal{\tilde W}^s}^2=\sum_{\ell \geq 1}\ell^{2s}\sum_{\lambda \in \Lambda_\ell} \langle g,g_\lambda\rangle^2 <\infty\Big\},\;\;s\in \R. \nonumber
\end{align}
For $\nu \geq 0$, Assumption \ref{spectral behaviour of K} implies that $K:{\mathcal W}^{-\nu/2} \rightarrow \mathcal {\tilde W}^{\nu/2}$
 is continuous.
In particular, when $\nu >0$, the operator $K$ is ill-posed 
with degree $\nu$, see for instance \cite{NP}. 
\subsection{Blockwise SVD reconstruction with noisy data} \label{def estimator}
Under Assumption \ref{blockwise SVD} and \ref{spectral behaviour of K}, we have the  reconstruction formula
\begin{equation} \label{reconstruction formula}
f = \sum_{\ell \geq 0} (K_{\ell})^{-1} \sum_{\lambda \in \Lambda_\ell} \langle Kf,g_\lambda\rangle\, e_\lambda.
\end{equation}
By the observed blurred version $K_{\delta}$ of $K$ in \eqref{Kdelta}, we obtain a 
family of estimators of $(K_{\ell})^{-1}$  from data \eqref{Kdelta} by considering the operator 
\begin{equation} \label{empirical operator}
{\bf 1}_{\big\{\|(K_{\delta,\ell})^{-1}\|_{G_\ell \rightarrow H_\ell} \leq \kappa\big\}}(K_{\delta,\ell})^{-1} ,
\end{equation}
where $\kappa >0$ is a cutoff level, possibly depending on $\ell$. Likewise, the coefficient  $\langle Kf,g_\lambda \rangle$ can be estimated by
\begin{equation} \label{empirical coeff}
z_{n,\lambda} := \langle y_n, g_\lambda\rangle.
\end{equation}
Mimicking the reconstruction formula \eqref{reconstruction formula} with the estimates \eqref{empirical coeff} and \eqref{empirical operator}, we finally obtain a (family of) estimator(s) of $f$ by setting
$$
\widehat f_{n,\delta} =
\sum_{0 \leq \ell \leq L} \big({K}_{\delta, \ell}\big)^{-1} \big(\sum_{\lambda \in {\Lambda}_{\ell}} z_{n,\lambda} e_\lambda  {\bf 1}_{\big\{\sum_{\lambda \in \Lambda_\ell} z_{n,\lambda}^2 \geq \, \tau_\ell^2\big\}} \big){\bf 1}_{{\mathcal E}_{\delta, \ell}(\kappa_\ell)}$$
where 
$${\mathcal E}_{\delta, \ell}(\kappa_\ell) = \big\{\|(K_{\delta,\ell})^{-1}\|_{G_\ell \rightarrow H_\ell} \leq \kappa_\ell\big\}.$$ 
The procedure is specified by the maximal frequency level $L$ and the threshold levels 
\begin{equation} \label{def kappa ell}
\kappa_\ell = \Big(\lambda_0 |\Lambda_\ell|^{-1/2}\big(\delta^2 |\log \delta|\big)^{-1/2} \Big) \bigwedge  n^{1/2}
\end{equation}
and
\begin{equation} \label{def tau ell}
\tau_\ell=\mu_0 |\Lambda_\ell|^{1/2} \big(n^{-1}\log n\big)^{1/2},
\end{equation}
for some prefactors $\lambda_0,\mu_0>0$. The threshold rule we introduce in both the signal (with level $\tau_\ell$) and the operator (with level $\kappa_\ell$) is inspired by classical block thresholding \cite{HKP, CZ, C} and will enable to adapt with respect to the smoothness properties of both the signal $f$ and the operator $K$, see below. 
\section{Main results} \label{main results}
\subsection{Minimax rates of convergence}
We assess the performance of the estimator $\widehat f_{n,\delta}$ defined in Section \ref{def estimator} over the Sobolev spaces linked to the basis $(e_\lambda, \lambda \in \Lambda)$ defined in \eqref{Sobolev}. 
Define the Sobolev balls ${\mathcal W}^s(M) = \{f\in {\mathcal W}^s,\;\|f\|_{{\mathcal W^s}}\leq M\}$ for $M>0$ and let  
\begin{equation} \label{def classe operator}
{\mathcal G}^{\nu}(Q)=\big\{K \in {\mathcal G},\;\;Q_i(K)\leq Q_i,i=1,2\big\}.
\end{equation}
for $Q=(Q_1,Q_2)$ with $Q_1>0$, $Q_1Q_2\geq 1$, where the mapping constants $Q_i(K)$ are defined in Assumption \ref{spectral behaviour of K}. 
\begin{theorem}[Upper bounds] \label{upper bounds} Let ${\mathcal G}$ be a class of operators satisfying Assumptions \ref{blockwise SVD} and \ref{spectral behaviour of K}. Assume we observe $(y_n,K_{\delta})$ given by \eqref{Gaussian deconvolution} and \eqref{Kdelta}, with $n \geq 1$ and $\delta \leq \delta_0 < 1$. Specify $\widehat f_{n,\delta}$ with 
\begin{equation} \label{choice of L}
L= \lfloor (\delta^2)^{-1/(2\nu+d-1)}\rfloor \bigwedge \lfloor n^{1/(2\nu+d)}\rfloor
\end{equation}
and $\kappa_\ell, \tau_\ell$ as in \eqref{def kappa ell} and \eqref{def tau ell}. 
For sufficiently small $\lambda_0$ and sufficiently large $\mu_0$, for every $s,M>0$, $Q=(Q_1,Q_2)$ with $Q_1>0$ and such that $Q_1Q_2\geq 1$, 
we have
\begin{align}
& \sup_{f\in {\mathcal W}^{s}(M),  K\in {\mathcal G}^\nu(Q)}
 \E\Big[\big\|\widehat f_{n,\delta}-f\big\|_{\bH}^2\Big] \nonumber \\ 
 \lesssim & \big(\delta^2 |\log\delta|\big)^{1 \bigwedge2s/(2\nu+d-1)}
\bigvee \big(n^{-1}\log n\big)^{2s/(2(s+\nu)+d)} \label{the upper bound}
\end{align}
where $\lesssim$ means inequality up to a multiplicative constant that depends on $d, s,\nu,M,Q$ and $\mu_0,\lambda_0$ only.
\end{theorem}
The bounds for $\mu_0$ and $\lambda_0$ are explicitly computable. In the model generated by $y_n$ in \eqref{Gaussian deconvolution} and $K_{\delta}$ in \eqref{Kdelta}, they depend on the dimension $d$ and on the absolute constants $c_0$ and $c_1$ of the concentration lemmas \ref{first ineq} and \ref{first ineq vector} below. However, they are in practice much too conservative, as is well known in the signal detection case \cite{DJ} or the classical inverse problem case \cite{AS}, see the numerical implementation Section \ref{examples}. 

Our next result shows that the rate achieved by $\widehat f_{n,\delta}$ is indeed optimal, up to logarithmic terms. The lower bound in the case $\delta = 0$ is classical (Nussbaum and Pereverzev \cite{NP}) and will not decrease for increasing noise levels $\delta$ or $n^{-1/2}$ whence it suffices to provide the case which formally corresponds to observing $Kf$ without noise while $K$ remains unknown.
\begin{theorem}[Lower bounds] \label{lower bounds}  In the same setting as Theorem \ref{upper bounds}, with in addition $Q_2>1/Q_1$, assume we observe $Kf$ exactly and $K_{\delta}$ given by \eqref{Kdelta}. For sufficiently small $\delta$, we have
\begin{equation} \label{borne inf}
\inf_{\widehat f} \sup_{f\in {\mathcal W}^{s}(M), K\in {\mathcal G}^\nu(Q)}\E\Big[\big\|\widehat f-f\big\|_{\bH}^2\Big] \gtrsim \big(\delta^2\big)^{1 \bigwedge2s/(2\nu+d-1)}
\end{equation}
where $\gtrsim$ means inequality up to a positive multiplicative constant that depends on $d, s,\nu, M$ and $Q$ only.

\end{theorem}
Combining \eqref{the upper bound} together with \eqref{borne inf} and the results of \cite{NP}, we conclude that $\widehat f_{n,\delta}$ is minimax over ${\mathcal W}^{s}(M)$ to within logarithmic terms in $n$ and $\delta$, and that this result is uniform over the nuisance parameter $K\in {\mathcal G}^\nu(Q)$.
\subsection{Discussion} \label{discussion}
\subsubsection*{The case of diagonal operators}
It is interesting to notice that the condition  $Q_2>1/Q_1$ in Theorem \ref{lower bounds} excludes the case where $K$ is diagonal.
In this particular case, considered especially in the deconvolution example of Section \ref{exemple deconvolution circulaire} below, a closer inspection of the proof of the upper bound shows that the rate 
$$n^{-s/(2(s+\nu)+d)} \bigvee \delta^{1 \bigwedge s/\nu}$$ can be obtained (up to some extra logarithmic factors) as in the case where $d=1$, which improves on the rate 
$$n^{-s/(2(s+\nu)+d)} \bigvee \delta^{1 \bigwedge 2s/(2\nu+d-1)}$$ provided by Theorem \ref{upper bounds}. This sheds some light on the role of the number $d$. It is in fact twofolds: it  acts as a 'dimension' in the term $n^{-2s/(2(s+\nu)+d)}$; in the term involving error in the operator $\delta$, it reflects  the distance to the diagonal case expanding from $ \delta^{1 \bigwedge s/\nu}$ in the diagonal case, to $\delta^{1 \bigwedge2s/(2\nu+d-1)}$ in the case $Q_2>1/Q_1$. It is very plausible, though beyond the scope of this paper, to  express conditions on $K$ leading to rates of the form $2s/(2\nu+\alpha)$, with $\alpha$ continuously varying from 0 to $d-1$. 
Note that in the case $d=1$, we recover  the minimax rate of density deconvolution with unknown error as proved by Neumann \cite{Neumann}, see also \cite{ComteLacour}.
\subsubsection*{Relation to other works in the case of sparse operators}  For an unknown signal $f$ with smoothness $s>0$ and unknown operator with degree of ill-posedness $\nu \geq 0$, the optimal rates of convergence are 
\begin{equation} \label{rate inverse}
n^{-\alpha(s,\nu)/2} \bigvee \delta^{\beta(s,\nu)},
\end{equation}
up to inessential logarithmic terms. The exponents $\alpha(s,\nu)$ and $\beta(s,\nu)$ are linked respectively to the error in the signal $y_n$ and the error in the operator $K_\delta$. Efromovitch and Kolchinskii \cite{EK} established that under fairly general conditions on the operator $K$, the optimal exponents are given by
$$\alpha(s,\nu) = \beta(s,\nu) = \frac{2s}{2(s+\nu)+d}.$$
They noted however that if certain sparsity properties on $K$ are moreover assumed (and that we shall not describe here, for instance if $K$ is diagonal in an appropriate basis) then the exponent $\beta(s,\nu) = \tfrac{2s}{2(s+\nu)+d}$ is no longer optimal, while $\alpha(s,\nu)$ remains unchanged. 

In the related context of operators acting on Besov spaces $B^s_{p,p}([0,1]^d)$ of functions with smoothness $s$ measured in $L^p$-norm,  Hoffmann and Rei\ss\, \cite{HR} introduce an {\it ad hoc} hypothesis on the sparsity of the unknown operator (that we shall not describe here either), expressed in terms of the wavelet discretization of $K$.
They subsequently obtain new rates of convergence for a certain nonlinear wavelet procedure, and these rates overperform \eqref{rate inverse} as expected from the results by \cite{EK}. In particular, if one considers the estimation of $f \in B^s_{2,2}$, in the extreme case where the operator $K$ is diagonal in a wavelet basis, 
%
the procedure in \cite{HR} achieves the rate
\begin{equation} \label{rate HR}
n^{-\alpha(s,\nu)/2} \bigvee (\delta^2)^{1 \bigwedge (s-d/2)/\nu}
\end{equation}
up to extra logarithmic terms.  We may compare our results with the rate \eqref{rate HR}. In our setting, if we pick $(e_\lambda, \lambda \in \Lambda)$ as the Fourier basis described in Section \ref{exemple deconvolution circulaire}, then we have ${\mathcal W}^s=B^s_{2,2}([0,1]^d)$. Assuming  $K$  to be diagonal in the basis $(e_\lambda, \lambda \in \Lambda)$ which is the exact counterpart of the approach of Hoffmann and Rei\ss\, with $K$ being diagonal in a wavelet basis, then by Theorem \ref{upper bounds}, 
 our estimator $\widehat f_{n,\delta}$ (nearly) achieves the rate 
$$n^{-\alpha(s,\nu)/2} \bigvee (\delta^2)^{1 \bigwedge 2s/(2\nu+d-1)}$$ 
which already outperforms the rate \eqref{rate HR} whenever the error in the signal $y_n$ is dominated by the error in the operator and $s$ is small compared to $\nu$, as follows from the elementary inequality
 $$2s/(2\nu+d-1) >  (s-d/2)/\nu\;\;\text{for}\;\;2\nu+d-1 \geq 2s.$$
The superiority of the blockwise SVD in this setting is explained by the fact that the wavelet procedure in \cite{HR} is agnostic to the diagonal structure of $K$ in the wavelet basis, in contrast to $\widehat f_{n,\delta}$ that takes full advantage of the block structure of $K$. 
As already explained in the preceding section, one could actually improve further this result in the specific case of $K$ being diagonal in $(e_\lambda, \lambda \in \Lambda)$ and show that $\widehat f_{n,\delta}$ (nearly) achieves the rate 
$n^{-\alpha(s,\nu)/2} \bigvee (\delta^2)^{1 \bigwedge s/\nu}$, thus deleting the `dimensional effect' of $d$ for the error in the operator.

\subsubsection*{Adaptation over the scales $\{{\mathcal W}^{s}, s>0\}$ and $\{{\mathcal G}^\nu,\nu \geq 0\}$}
The estimator $\widehat f_{n,\delta}$ is fully adaptive over the family of Sobolev balls $\{W^{s}(M),s>0,M>0\}$ (in the sense that $\widehat f_{n,\delta}$ does not require the knowledge of $s$ nor $M$). However,  the knowledge of the degree of ill-posedness $\nu$ of $K$ is required through the choice of the maximal frequency $L$ in \eqref{choice of L}. This restriction can actually be relaxed further in dimension $d\geq 2$. Indeed,  setting formally $\nu=0$ in \eqref{choice of L}, one readily checks that $\widehat f_{n,\delta}$ becomes adaptive over $\{{\mathcal W}^{s}(M), s>0,M>0\}$ and $\{{\mathcal G}^\nu(Q),\nu \geq 0,Q=(Q_1,Q_2), Q_1Q_2 \geq 1\}$ simultaneously. In dimension $d=1$ however, setting $\nu=0$ in \eqref{choice of L} is forbidden, but an alternative adaptivity result can be obtained by taking
$L =\lfloor (\delta^2)^{-1/s_0}\rfloor \wedge n$ for some $s_0>0$, in which case $\widehat f_{n,\delta}$ is fully adaptive over the scale $\{{\mathcal G}^\nu(Q),\nu \geq 0,Q=(Q_1,Q_2), Q_1Q_2 \geq 1\}$ and the restricted family $\{W^{s}(M),s \geq s_0,M>0\}$. 

\subsubsection*{Extension to density estimation}
We briefly show a line of proof for extending Theorem \ref{upper bounds} to the framework of density estimation. Suppose that instead of $y_n$ we observe a random sample  $Z_1,\ldots, Z_n$ drawn from $Kf$ assumed to a probability density.
 By analogy to \eqref{empirical coeff}, we have an estimator of $\PP(g_{\lambda}) = \langle Kf,g_{\lambda}\rangle$ replacing $\langle y_n, g_\lambda \rangle$ with 
$$\PP_n(g_{\lambda})=n^{-1}\sum_{i =1}^n g_\lambda(Z_i).$$
Writing
$$\PP_n(g_{\lambda}) = \langle Kf,g_{\lambda}\rangle + n^{-1/2} \eta_{n,\lambda},$$
with  $\eta_{n,\lambda} = n^{1/2}\big(\PP_n(g_{\lambda})- \PP(g_{\lambda})\big)$,
an inspection of the proof of Theorem \ref{upper bounds} reveals that an extension to the density estimation setting carries over as soon as the vector
$(\eta_{n,\lambda},\lambda \in \Lambda_{\ell})$
satisfies a concentration inequality, namely
$$\exists \beta_1 >0,c_1>0,\;\forall \beta \geq \beta_0,\;\;\PP\Big(|\Lambda_\lambda|^{-1}\sum_{\lambda \in \Lambda}\eta_{n,\lambda}^2\ge \beta^2\Big) \leq \exp\big(-c\beta^2|\Lambda_\ell|\big),$$
see \eqref{first ineq vector} in Lemma \ref{concentration vector gaussian}.
To that end, we may apply a concentration inequality by Bousquet \cite{B} as developed for instance in Massart \cite{M}, Eq (5.51) p. 171. The precise control of this extension requires further properties on the basis $(g_{\lambda}, \lambda \in \Lambda)$ and on the density $Kf$ via the behaviour of $\sum_{\lambda \in \Lambda_\ell}\text{Var}\big(g_\lambda(Z_1)\big)$, see Eq. (5.52) p. 171 in \cite{M}. We do not pursue that here. 
\section{Application to blind deconvolution} \label{examples}

\subsection{Spherical deconvolution} \label{exemple deconvolution spherique}
\begin{proof}[Scientific context]
A common challenge in astrophysics is the analysis of complex data sets
consisting of a number of objects or events such as galaxies of a particular
type or ultra high energy cosmic rays (UHECR) and that are genuinely 
distributed over the celestial sphere. Such objects or events 
are distributed according to a probability density distribution $f$ on the sphere,
depending itself on the physics that governs the production of these objects or events.
For instance, UHECR are particles of unknown nature  arriving
at the earth from apparently random directions of the sky. They
could originate from long-lived relic particles from the Big Bang. Alternatively, they could be generated
by the acceleration of standard particles, such as protons, in
extremely violent astrophysical phenomena.
They could also originate from Active Galactic Nuclei (AGN), or from
neutron stars surrounded by extremely high magnetic fields.
As a consequence, in some hypotheses, the underlying probability distribution for observed UHECRs would be a finite sum of point-like sources. In other hypotheses, the distribution could be uniform, or smooth and correlated with the local distribution of matter in the universe. The distribution could also be a superposition of the above. Identifying between these hypotheses is of primordial importance for understanding the origin and mechanism of production of UHECRs.
The observations, denoted by $X_i$, 
are often perturbated by an experimental noise, say $\ep_i$, 
that lead to the deconvolution problem described in Section \ref{motivation}. Following van Rooj and Ruymgart \cite{RR}, Healy {\it et al.} \cite{Healy1998}, Kim and Koo \cite{KK} and 
Kerkyacharian {\it et al.} \cite{KPP}, we assume the following model: we observe an $n$-sample
$(Z_1,\ldots, Z_n)$ with
$$Z_i = \varepsilon_i X_i,\;\;i=1,\ldots, n$$
where the $X_i$ are distributed on the sphere $\Sph$, with common density $f$ with respect to  the uniform probability distribution $\mu(d\omega)$ on $\Sph$ and independent of the $\varepsilon_i$ that have a common density $g$ with respect to the Haar probability measure $dr$ on the group $\SO$ of $3 \times 3$ rotation matrices. One proves in \cite{Healy1998, KK} that the density of the $Z_i$ is  
\begin{equation} \label{spherical convolution}
Kf(\omega) =g\star f(\omega) := \int_{\SO}g(r)f(r^{-1}x)dr,\;\;\omega\in \Sph
\end{equation}
and we are interested in the case where the exact form  $g$ of the convolution operator $K = g \star \cdot$ is unknown, due for instance to insufficient knowledge of the device that is used to measure the observations. However, we assume approximate knowledge of $g$ through $K_\delta$ as defined in \eqref{Kdelta}.
 \end{proof}
\begin{proof}[Checking the blockwise SVD Assumptions \ref{blockwise SVD} and \ref{blockwise SVD}]
We closely follow the exposition of \cite{Healy1998, KK, KPP} for an overview of Fourier theory on $\Sph$ and $\SO$ in order to establish rigorously the connection to Theorem \ref{upper bounds} and \ref{borne inf}.
Define 
$$u(\phi)=\begin{pmatrix} \cos \phi & -\sin \phi & 0 \\ \sin \phi & \cos \phi &0 \\ 0 &0 & 1 \end{pmatrix}\;\;\text{and}\;\;a(\theta)=\begin{pmatrix} \cos \theta & 0 & \sin \theta \\  0 & 1 &0 \\ - \sin \theta  &0 & \cos \theta \end{pmatrix}$$
where $\phi \in [0,2\pi),\,\theta\in[0,\pi)$. Every rotation $r \in \SO$ has representation 
$r=u(\phi) a (\theta) u(\psi)$ for some $\phi,\psi \in [0,2\pi), \theta\in[0,\pi)$. Define the rotational harmonics 
$$
D^l_{mn}(r)=D^l_{mn}(\phi,\theta,\psi)=e^{-i(m\phi+n\psi)}P^l_{mn}\big(\cos(\theta)\big)
$$
for $l\in\mathbb{N}, -l\le m,n\le l
$
where $P^l_{mn}$ are the second type Legendre functions described in details in \cite{V}.
The $D^l_{mn}$ are the eigenfunctions of the Laplace-Beltrami operator on $\SO$ hence the family $(\sqrt{2l+1}D^l_{mn})$ forms a complete orthonormal basis of $L^2(dr)$ on $\SO$, where $dr$ is the Haar probability measure. Every $h\in L^2(dr)$ has a rotational Fourier transform   
$$
{\mathcal F}(h)^l_{mn} =\int_{\SO} h(u) D^l_{mn}(u) du,\;\;l\in\mathbb{N}, -l\le m,n\le l,
$$
and for every $h \in L^2(dr)$ we have a reconstruction formula
\begin{align*}
h& =\sum_{l \in \mathbb{N}} \sum_{-l \leq m,n \leq l} {\mathcal F}(h)^l_{mn} \overline{D^l_{mn}} \\
& =\sum_{l \in \mathbb{N}} \sum_{-l \leq m,n \leq l} {\mathcal F}(h)^l_{mn} D^l_{mn}(\cdot^{-1})
\end{align*}
An analogous analysis is available on $\Sph$. Any point $\omega \in \Sph$ is determined by its spherical coordinates
$\omega=\big(\sin(\theta)\cos(\phi),\sin(\theta)\sin(\phi),\cos(\theta)\big)$ for some $\theta \in [0,\pi), \phi \in [0,2\pi)$. Define 
\begin{equation}
Y^l_m(\omega)=Y_l^m(\theta,\varphi)=(-1)^m \sqrt{\tfrac{2l+1}{4\pi}\tfrac{(l-m)!}{(l+m)!}}P^l_m\big(\cos(\theta)\big)e^{im\varphi}
\end{equation}
for $l\in\mathbb{N}, -l\le m\le l$ where $P^l_m$ are the Legendre functions. We have $Y^l_{-m}=(-1)^mY^l_m$ and the $(Y_m^l)$ constitute an orthonormal basis of $L^2(\mu)$ on $\Sph$, generally referred to as the spherical harmonic basis. Any $f\in L^2(\mu)$ has a spherical Fourier transform
$${\mathcal F}(f)_m^l=\int_{\Sph}f(\omega)\overline{Y_m^l(\omega)}\mu(d\omega)$$
and a reconstruction formula
$$f=\sum_{\ell \in \mathbb{N}}\sum_{-l\leq m\leq l}{\mathcal F}(f)_m^lY_m^l.$$
If $g\in L^2\big(\SO\big)$ the spherical convolution operator $Kf=g\star f$ defined in \eqref{spherical convolution} 
satisfies
\begin{equation} \label{blockSVDsphere}
{\mathcal F}(g\star f)_m^l=\sum_{n=-l}^l{\mathcal F}(g)_{mn}^l{\mathcal F}(f)_n^l
\end{equation}
and we retrieve the blockwise SVD formalism of Section \ref{blockwise SVD property} in dimension $d=2$ by setting $\bH=\G=L^2(\Sph,\mu)$,  where $\mu$ the probability Haar measure on $\Sph$ and
$$e_\lambda =g_\lambda= Y_m^\ell\;\;\text{with}\;\;\lambda = (m,\ell),\;\; \Lambda_\ell = \{(m,\ell),\;\;-\ell\le m\le \ell\}.$$
We have $|\Lambda_\ell| = 2\ell+1$ and by \eqref{blockSVDsphere}, $K_\ell$ is the finite dimension operator stable on $\text{Span}\{e_\lambda, \lambda \in \Lambda_\ell\}$ with matrix having entries
$$(K_\ell)_{mn}={\mathcal F}(g)_{mn}^\ell.$$
Hence Assumption \ref{blockwise SVD} is satisfied.  Notice also that in this case $K_\ell$ is generally not diagonal.
Assumption \ref{spectral behaviour of K} is satisfied as we assume that $g$ is {\it ordinary smooth} in the terminology of Kim and Koo \cite{KK}. Our Assumption \ref{spectral behaviour of K} exactly matches the constraint (3.6) in their paper
with examples given by the Laplace distribution on the sphere ($\nu=2$) or the Rosenthal distribution ($\nu>0$ arbitrary).
\end{proof}
\begin{proof}[Numerical implementation] Following Kerkyacharian, Pham Ngoc and Picard \cite{KPP} in their  Example 2, we take $f(\omega) = C\exp(-4\|\omega-\omega_1\|^2)$ with $\omega_1=(0,1,0)$ and $C=1/0.7854$. We have $\|f\|_{L^2(\mu)}=0.7469$.\\
$g$ is the density of a Laplace distribution on $\SO$, defined through  ${\mathcal F}(g)_{mn}^\ell=\delta_{mn} \big(1+\ell(\ell+1)\big)^{-1}$. Hence, the matrices $(K_\ell)_{mn}$ are homotheties whose ratios behave as $\ell^{-2}$. We have $\nu=2$.\\
 We plot in Figures \ref{Figure2} a $1000$-sample of $X_i$ with density $f$ on the sphere, and the action by $\varepsilon_i$ on the $X_i$, where the $\varepsilon_i$ are distributed according to $g$ in Figure \ref{Figure3}. Note that for the estimation of $g$, we have access to a noisy version of $g$ with noise level $\delta$ only. 
\begin{center}
\begin{figure}
\centering
\includegraphics[width=8cm, height=8cm]{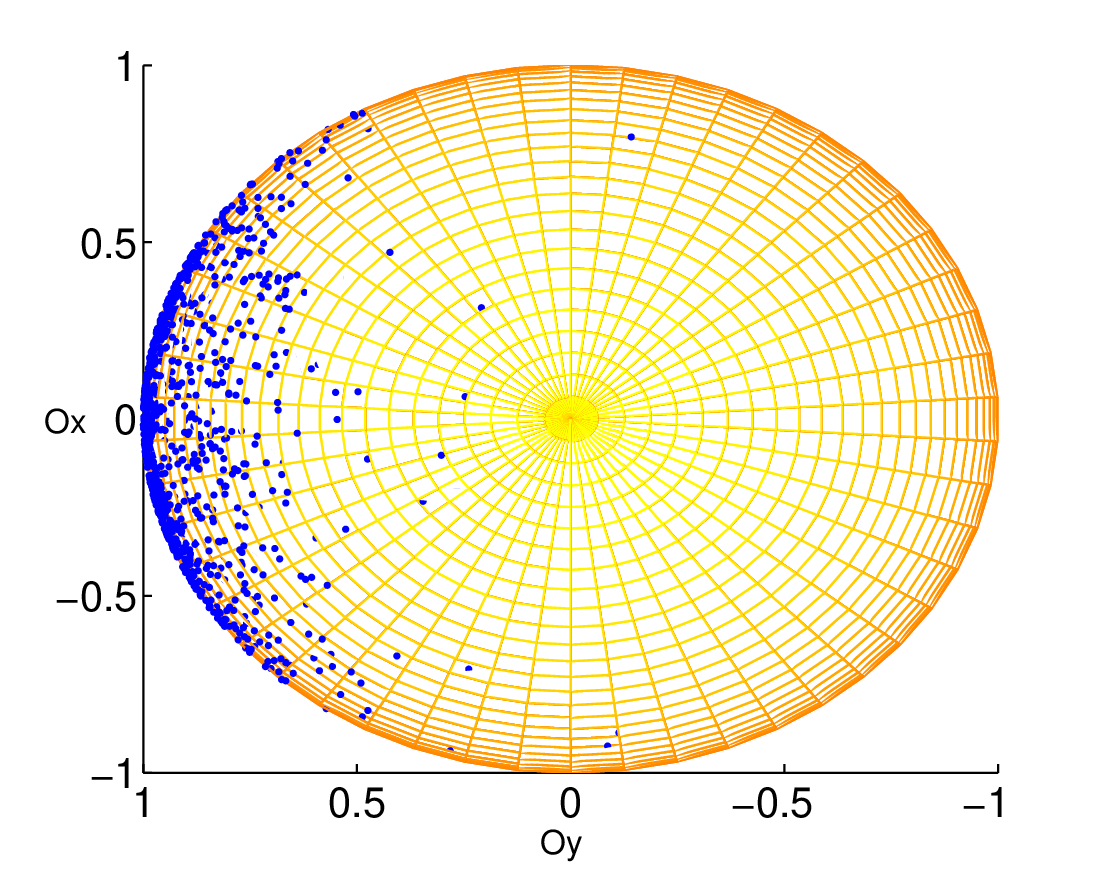}
\caption{\label{Figure2} \footnotesize{\textbf{Data from $f$.} Plot of $n=1000$ data with common distribution $f$ on the sphere $\mathbb{S}$ (planar representation).}}
 \end{figure}
\begin{figure}
\centering
\includegraphics[width=8cm, height=8cm]{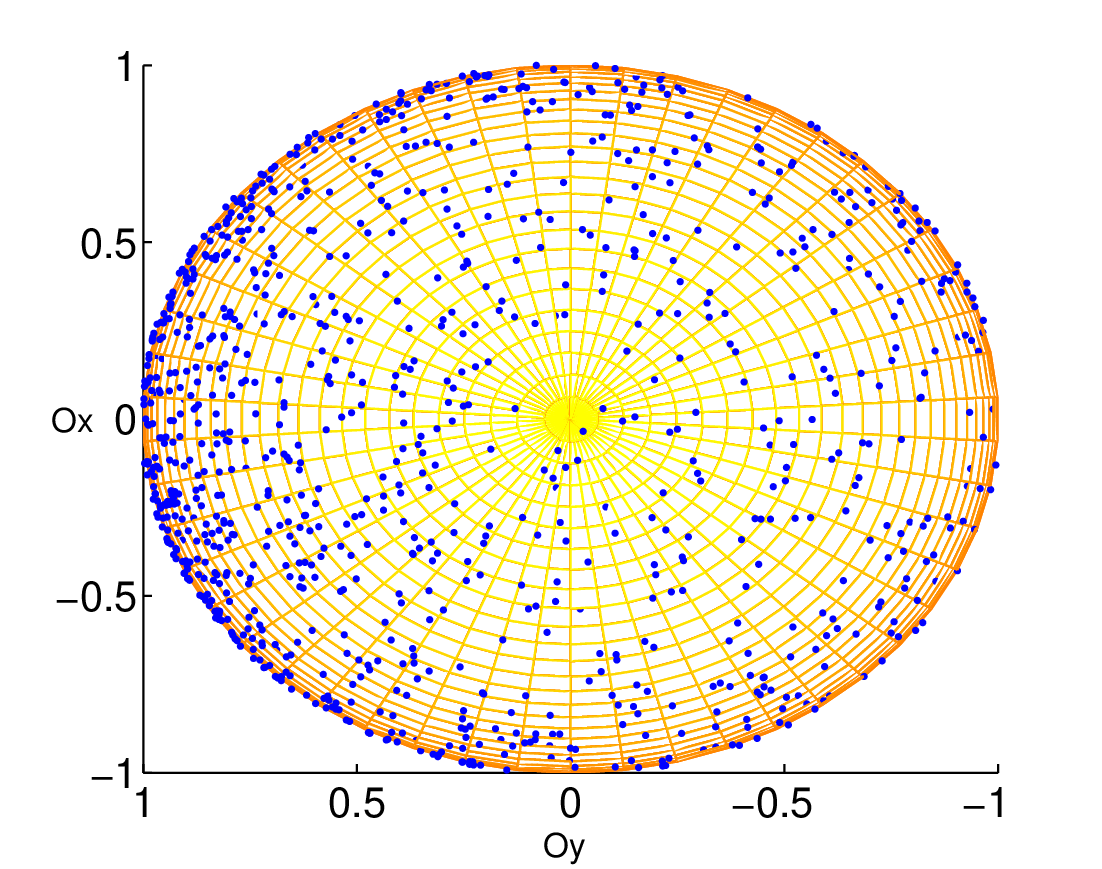}
\caption{\label{Figure3} \footnotesize{\textbf{Data from $f\star g$.} Plot of $n=1000$ data $\epsilon_i X_i$ on the sphere $\mathbb{S}$ with common distribution $Kf = f\star g$. The $X_i$ are the data pictured in Figure \ref{Figure2} and the $\varepsilon_i$ are sampled according to $g$ (planar representation).}}
 \end{figure}
 \end{center}
We display below the (renormalised) empirical squared error of $\widehat f_{10^8,\delta}$ (oracle choice $\lambda_0=1, \mu_0=1$) for $1000$ Monte-Carlo for several values of $\delta$. The noise level $\delta$ is to to be compared with the noise level $n^{-1/2} = 10^{-4}$. The latter is chosen non-negative, in order to show the interaction between the two types of error, and sufficiently small to emphasize the influence of $\de$ on the process of estimation.
\begin{center} 
\begin{tabular}{| l | c | c | c | c | c | }\hline
{\small Noise level} $\delta$ & $0$ & $10^{-3}$ & $3\,10^{-3}$ & $5\,10^{-3}$ & $10^{-2}$   \\
\hline
${\small \text{Mean error}}$ & 0.0466 & 0.0542 & 0.1732  & 0.2784 & 0.4335 \\
\hline
{\small Standard dev.}& 0.0011&0.0022&0.0126&0.0355&0.0466 \\
\hline 
\end{tabular}
\end{center}

Finally, on a specific sample of $n=10^8$ data, we plot the target density $f$ (Figure \ref{Figure4}) and its reconstruction for $n=10^8$ data with $\delta=0$ (Figure \ref{Figure5}) and $\delta = 3\,10^{-3}$ (Figure \ref{Figure6}). At a visual level, we oversimplify the representation by plotting $f$ and its reconstruction with a view from above the sphere through the $Oz$ axis. We see that the contour in Figure \ref{Figure6} is not well recovered in the regions where $f$ is small (on the right side of the graph in Figure \ref{Figure6}). The choice of $\lambda_0,\mu_0$ remains unchanged.
\begin{figure}
\includegraphics[width=12.5cm, height=6cm]{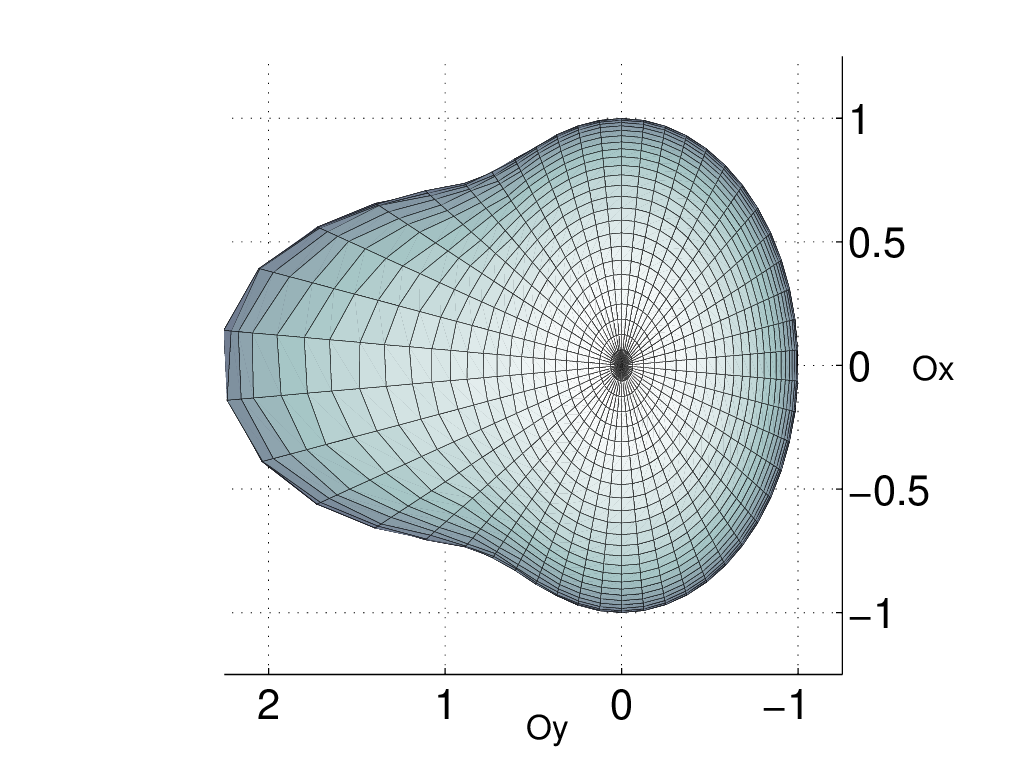}
\caption{\label{Figure4} \footnotesize{\textbf{Target density $f$.} The representation is simplified through a view from above the sphere through the $Oz$-axis.}}
 \end{figure}
\begin{figure}
\includegraphics[width=12.5cm, height=6cm]{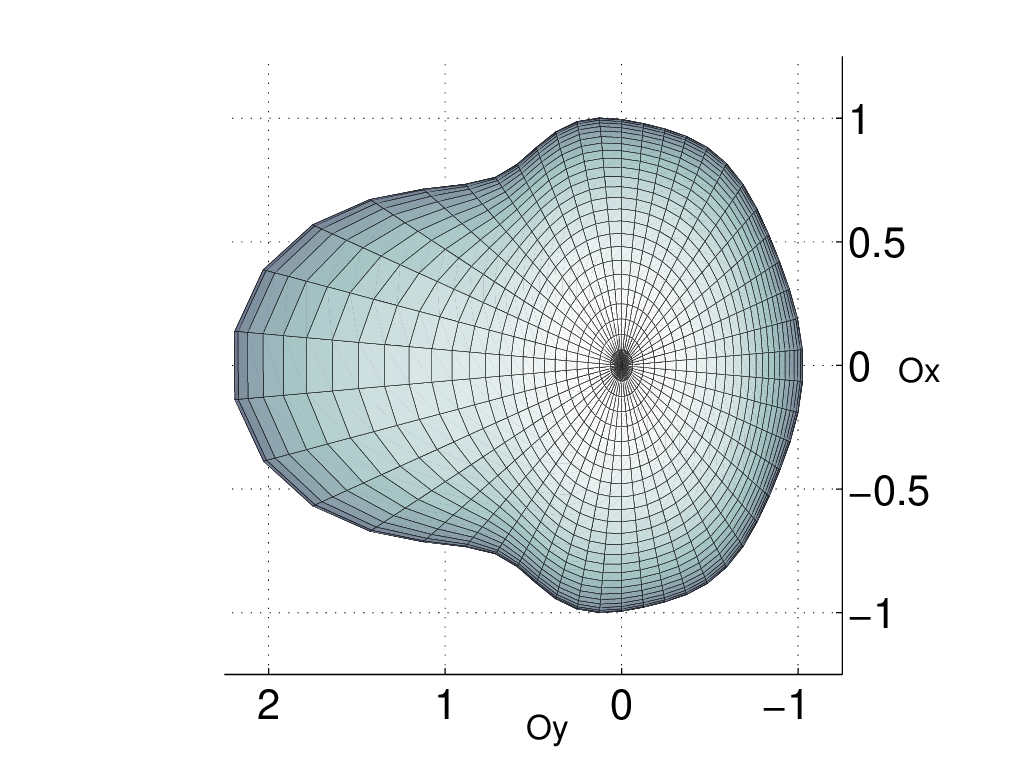}
\caption{\label{Figure5} \footnotesize{\textbf{Reconstruction for $n=10^8$ and $\delta=0$}.}}
 \end{figure}
\begin{figure}
\includegraphics[width=12.5cm, height=6cm]{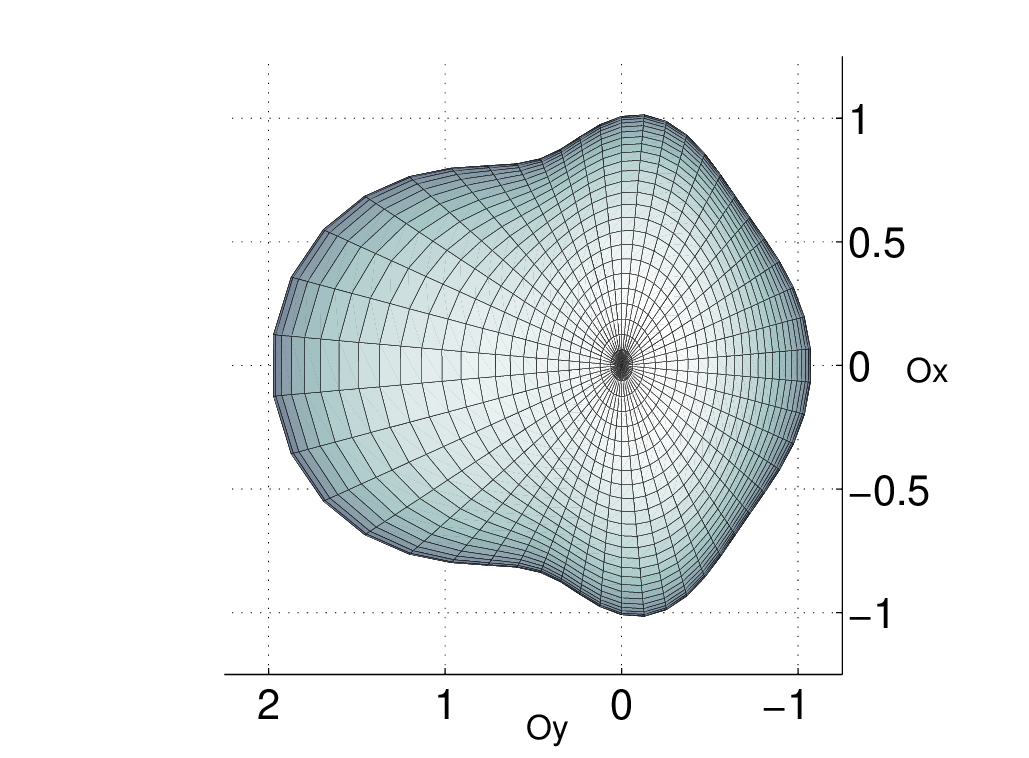}
\caption{\label{Figure6} \footnotesize{\textbf{Reconstruction for $n=10^8$ and $\delta=3\,10^{-3}$}. The reconstruction is polluted simultaneously by the limited number of observations $n$ and the noise level $\delta$ in the blurring $g$.}}
 \end{figure}
\end{proof}
\subsection{Circular deconvolution} \label{exemple deconvolution circulaire}

\begin{proof}[Scientific context] In many engineering problems, the observation of a signal $f$ or image is distorted by the action of a linear operator $K$. We assume for simplicity that $f$ lives on the torus $\mathbb{T}=[0,1]$ (or $[0,1]^d$) appended with periodic boundary conditions. In many instances, the restoration of $f$ from the noisy observation of $Kf$ is challenged by the additional uncertainty about the operator $K$. This is the case for instance in electronic microscopy \cite{microscopy} for the restoration of fluorescence Confocal Laser Scanning Microscope (CLSM) images. In other words, the quality of the image suffers from two physical limitations: error measurements or limited accuracy, and the fact that the exact PSF (the incoherent point spread function) that accounts for the blurring of $f$ (mathematically the action of $K$) is not precisely known. This is a classical issue that goes back to \cite{PSF1, PSF2}. An idealised additive Gaussian model for the noise contamination yields the observation \eqref{Gaussian deconvolution}
with 
$$Kf(x)=g \star f(x) := \int_{{\mathbb T}^d} g(u)f(x-u)du,\;\;x\in {\mathcal D} = {\mathbb T}^d.$$
The degradation process $K = g \star \cdot$ is characterised by the impulse response function $g$. In most cases of interest, we do not know the exact form of $g$. In a condensed idealised statistical setup, we have access to
\begin{equation} \label{Kdelta function}
g_{\delta} = g+\delta \dot W',
\end{equation}
where $\dot W'$ is another Gaussian white noise defined on $ L^2(\mu)$ and independent of $\dot W$. Experimental approaches that justify representation \eqref{Kdelta function} are described in \cite{PSFEst1, Keller, Hell}.
\end{proof}
\begin{proof}[Checking Assumptions \ref{blockwise SVD} and \ref{blockwise SVD}] We obviously have $\bH=\G=L^2(\mathbb{T}^d)$ and the bases $(e_\lambda)$ and $(g_\lambda)$ will coincide with the $d$-dimensional extension of the circular trigonometric basis $(e^{2i\pi k x}, k \in \mathbb{Z})$ if we set:
$$e_\lambda(x_1,\ldots,x_d)=\prod_{j=1}^d e^{2i\pi k_j x_j},\;\;(x_1,\ldots, x_d) \in \mathbb{T}^d,$$
where we put
$$\lambda = (k_1,\ldots, k_d),\;\;\ell = |\lambda| = 1+ \sum_{j=1}^d|k_j|,\;\;\text{and}\;\;\ell \geq 1.$$
Any $f\in L^2(\mu)$ has a Fourier transform ${\mathcal F}(f)_\lambda = \int_{\mathbb{T}^d}f(x)\overline{e_\lambda(x)}\mu(dx)$ and moreover, if $g \in L^2(\mu)$, we have
$${\mathcal F}(f\star g)_\lambda = {\mathcal F}_\lambda(f){\mathcal F}_\lambda(g).$$
Therefore, $K$ is diagonal in the basis $(e_{\lambda}, \lambda \in \Lambda)$ henceforth stable. Moreover, with $\Lambda_\ell = \{\lambda, |\lambda|=\ell\}$, we have $\abs{\Lambda_\ell}=\begin{pmatrix} \ell-1+d \\ d-1 \end{pmatrix} \sim \ell^{d-1}$.  Moreover $K_\ell=\mathrm{Diag}\big({\mathcal F}_\lambda(g),\lambda \in \Lambda_\ell\big)$ and Assumption \ref{blockwise SVD} follows. Assuming that $g$ satisfies $c |\lambda|^{-\nu} \leq \big|{\mathcal F}(g)_\lambda \big| \leq c' |\lambda|^{-\nu}$ for some $\nu \geq 0$ and constants $c,c'>0$, we readily obtain Assumption \ref{spectral behaviour of K}. Note also that since $K$ is diagonal in the basis $(e_\lambda, \lambda \in \Lambda)$ observing $g_\delta$ in the representation \eqref{Kdelta function} is equivalent to observing $K_\delta$ in \eqref{Kdelta}.
\end{proof}
\begin{proof}[Numerical implementation]
We numerically implement $\widehat f_{n,\delta}$ in dimension $d=1$ in the case where there is no noise in the signal (formally $n^{-1/2}=0$) in order to illustrate the parametric effect that dominates in the optimal rate of convergence in Theorems \ref{upper bounds} and \ref{borne inf} that becomes $\big(\delta^2\big)^{s/\nu \wedge 1}$ in that case. We take as target function $f : \mathbb{T} \rightarrow \R$ belonging to ${\mathcal W}^{5-\alpha}$ for all $\alpha >1/2$ and defined by its Fourier coefficients $${\mathcal F}(f)_\lambda = |\lambda|^{-5},\;\;\lambda \in \{-1000,\ldots, 1000\}.$$
We pick a family of blurring functions $g_{\nu}$ defined in the same manner by the formula
$${\mathcal F}(g_\nu)_\lambda = |\lambda|^{-\nu},\;\;\lambda \in \{-1000,\ldots, 1000\},\;\;\nu \in \{1, 4, 5, 6, 8\}.$$
\begin{figure}
\includegraphics[width=10cm,height=5cm]{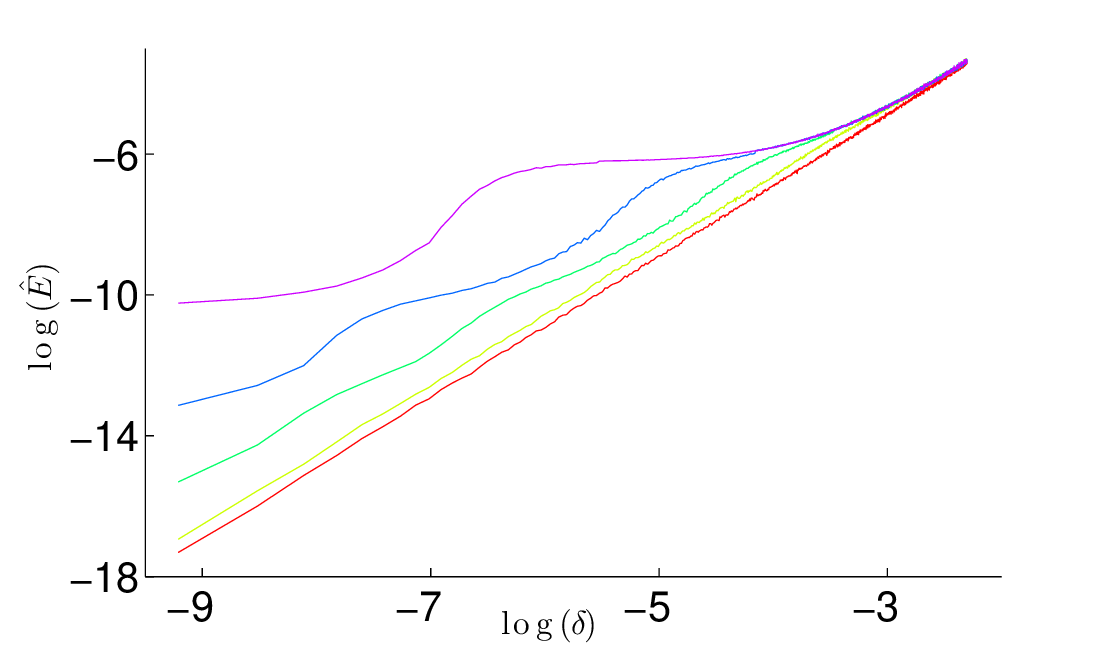}
\caption{\label{Figure1} \footnotesize{\textbf{Estimation of the rate exponent when $n^{-1/2} \ll \delta$.} Empirical squared-error $\hat{E}$ versus $\delta$ in log-log scale.  Top-to-bottom: $\nu=8,6,5,4,1$. The target function has smoothness $s=5-\alpha$ for all $\alpha >1/2$. For $\nu < 4.5$, the slope of the curve is constant and close to $2$, confirming the parametric rate predicted by the theory when the smoothess of the signal dominates the degree of ill-posedness of the operator. The empirical errors were computed using $1000$ Monte-Carlo simulations.}}
 \end{figure}
We show in Figure \ref{Figure1} in a log-log plot the mean-squared error of $\widehat f_{\infty,\delta}$ for the oracle choice $\mu_0=0, \lambda_0=1$ over 1000 Monte-Carlo simulations for $\nu \in \{1,4,5, 6,8\}$ and $\delta \in [10^{-4}, 10^{-1}]$. For small values of $\delta$ the predicted slope of the curve gives a rough estimate of the rate of convergence. We visually see that for the critical case $\nu \leq s=5-\alpha$ with $\alpha >1/2$ and below, the slope is close to $2$ confirming the parametric rate that is obtained whenever $\nu \leq s$. 
\end{proof}

\section{Proofs} \label{proofs}
\subsection{Preliminary estimates}
\begin{proof}[Preparation]
Recall 
that $H_\ell = \mathrm{Span} \{e_\lambda, \lambda \in \Lambda_\ell\}$, $G_\ell = \mathrm{Span} \{g_\lambda, \lambda \in \Lambda_\ell\}$ and that $P_\ell$ (resp. $Q_\ell$) denotes the orthogonal projector onto $G_\ell$ (resp. $H_\ell$).
For $h \in \bH $, we have 
$$P_\ell Kh = P_\ell KQ_\ell h + P_\ell K(\text{Id}-Q_\ell)h.$$
Using Assumption \ref{blockwise SVD}, we have
$K(\text{Id}-Q_\ell)h \in G_\ell^\perp$ and therefore $P_\ell K(\text{Id}-Q_\ell)h=0$. As a consequence
\begin{equation} \label{stabilite K}
P_\ell K  = K_\ell Q_\ell.
\end{equation}
In turn, we have a convenient description of the observation $K_\delta$ defined in \eqref{Kdelta} and $y_n$ defined in \eqref{Gaussian deconvolution} and in terms of a sequence space model that we shall now describe.
\end{proof}
\begin{proof}[Notation]
If $h \in \G$, we denote by 
${\boldsymbol h}_\ell$ the (column) vector of coordinates of $P_\ell h$ in the basis $(g_\lambda, \lambda \in \Lambda_\ell)$.
If $T:\bH\rightarrow \G$ is a linear operator, we write ${\boldsymbol T}_\ell$ for the matrix of the Galerkin projection $T_\ell=P_\ell T_{|H_\ell}$ of $T$. 
\end{proof}
\begin{proof}[Sequence model for error in the operator] The observation of $K_{\delta}$ in \eqref{Kdelta} leads to the representation
$K_{\delta,\ell} = K_\ell+\delta \dot B_\ell$, or equivalently, in matrix notation
\begin{equation} \label{matrix Kdelta}
\boldsymbol{K}_{\delta, \ell} = \boldsymbol{K}_\ell + \delta\boldsymbol{\dot B}_\ell,\;\;\ell \geq 1,
\end{equation}
where $\boldsymbol{\dot B}_\ell$ is a $|\Lambda_\ell|\times |\Lambda_\ell|$ matrix with entries that are independent centred Gaussian random variables, with unit variance. The following estimate is a classical concentration property of random matrices. For $\ell \leq L$,
  $\|\cdot\|_{\text{op}}$ denotes the operator norm for $|\Lambda_\ell| \times |\Lambda_\ell|$ matrices (we shall skip the dependence upon $\ell$ in the notation). 
\begin{lemma}[\cite{DS}, Theorem II.4] \label{concentration operator}There are positive constants $\beta_0,c_0$ such that
\begin{align}
\text{For all}\;\;\beta \geq \beta_0,&\;\PP\big(|\Lambda_\ell|^{-1/2}\|\boldsymbol{\dot B}_\ell\|_{\text{op}}\geq \beta\big) \leq \exp\big(-c_0\beta^2|\Lambda_\ell|^2\big). \label{first ineq}
\end{align}
\end{lemma}
An immediate consequence of Lemma \ref{concentration operator} is the following moment bound:
\begin{equation} \label{moment bound norm operator}
\text{For every}\;\;p>0,\;\;\E\big[\|\boldsymbol{\dot B}_\ell\|_{\text{op}}^p\big] \lesssim |\Lambda_\ell|^{p/2}.
\end{equation}
\end{proof}
\begin{proof}[Sequence model for error in the signal] From \eqref{Gaussian deconvolution}, we observe the Gaussian measure $y_n$, or equivalently, thanks to \eqref{stabilite K}
\begin{align*}
P_\ell y_n = P_\ell Kf+n^{-1/2} P_\ell \dot W = K_\ell Q_\ell f + n^{-1/2}\boldsymbol{\eta}_\ell,\;\;\ell \geq 1
\end{align*}
or, using the notation introduced in \eqref{empirical coeff}, in matrix notation 
\begin{equation} \label{matrix gaussian deconvolution}
{\boldsymbol z}_{n,\ell}  = \boldsymbol{K}_\ell \boldsymbol{f}_\ell +n^{-1/2}\boldsymbol{\eta}_\ell,\;\;\ell \geq 1
\end{equation}
where we used \eqref{stabilite K}, with $\boldsymbol{\eta}_\ell$ denoting a vector of $|\Lambda_\ell|$ independent centred Gaussian random variables with unit variance. 

The following result is a direct consequence of the fact that $\boldsymbol{\eta}_\ell$ has a $\chi$-square distribution with $|\Lambda_\ell|$ degrees of freedom. The proof is standard
\begin{lemma} \label{concentration vector gaussian}
There are positive constant $\beta_1,c_1$ such that
\begin{equation} \label{first ineq vector}
\text{For all}\;\;\beta \geq \beta_1,\;\PP\big(|\Lambda_\ell|^{-1/2}\|\boldsymbol{\eta}_\ell\|\geq \beta\big) \leq \exp\big(-c_1\beta^2|\Lambda_\ell| \big), 
\end{equation}
\end{lemma}
\end{proof}
\subsection{Proof of Theorem \ref{upper bounds}}
We have
$$\|\widehat f_{n}-f\|_{\bH}^2  =  \sum_{\ell \geq 1} \|\widehat {\boldsymbol f}_{n,\ell}-\boldsymbol{f}_\ell\|^2 = \sum_{\ell = 1}^L\|\widehat {\boldsymbol f}_{n,\ell}-\boldsymbol{f}_\ell\|^2+\sum_{\ell > L}\|\boldsymbol{f}_\ell\|^2 
$$
where $\|\cdot\|$ denotes the Euclidean norm on $\R^{|\Lambda_\ell|}$ (we shall omit any reference to $\ell$ when no confusion is possible). Concerning the bias term, we have 
\begin{equation} \label{controle biais}
\sum_{\ell > L}\|\boldsymbol{f}_\ell\|^2\leq\|f\|_{{\mathcal W}^s}^2L^{-2s}
\end{equation}
and this term has the right order by definition of $L$ in \eqref{choice of L}.
Concerning the stochastic term, thanks to our preliminary analysis, we may write
$$\widehat {\boldsymbol f}_{n,\ell} = (\boldsymbol{K}_{\delta,\ell} )^{-1}\boldsymbol{z}_{n,\ell} {\bf 1}_{\{\|(\boldsymbol{K}_{\delta,\ell})^{-1}\|_{\text{op}} \leq \kappa_\ell\}} {\bf 1}_{\{\|\boldsymbol{z}_{n,\ell}\|\geq \tau_\ell\}},$$
We set
$${\mathcal A}_\ell  = \{\|(\boldsymbol{K}_{\delta,\ell})^{-1}\|_{\text{op}}\leq \kappa_\ell\}\;\;\text{and}\;\;{\mathcal B}_\ell  = \{\|\boldsymbol{z}_{n,\ell}\|\geq \tau_\ell\}.$$
We thus obtain the decomposion of the variance term as 
$$\sum_{\ell=1}^L\|\widehat {\boldsymbol f}_{n,\ell} -\boldsymbol{f}_\ell \|^2 \leq I+II+III,$$
with
\begin{align*}
I & =\sum_{\ell=1}^L\|(\boldsymbol{K}_{\delta,\ell})^{-1}\boldsymbol{z}_{n,\ell} -\boldsymbol{f}_\ell\|^2  
{\bf 1}_{{\mathcal A}_\ell}{\bf 1}_{{\mathcal B}_\ell}\\
II & =\sum_{\ell=1}^L\|\boldsymbol{f}_\ell\|^2 {\bf 1}_{{\mathcal A}_\ell^c},\\
III & =\sum_{\ell=1}^L\|\boldsymbol{f}_\ell\|^2 {\bf 1}_{{\mathcal B}_\ell^c}. 
\end{align*}
We shall successively bound each term $I$, $II$ and $III$.
\begin{proof}[$\bullet$ The term I, preliminary decomposition] Writing
$$\boldsymbol{z}_{n,\ell} = \big(\boldsymbol{K}_{\delta,\ell}-\delta \boldsymbol{\dot B}_\ell\big)\boldsymbol{f}_\ell + n^{-1/2}\boldsymbol{\eta}_\ell,$$
we obtain
$$ (\boldsymbol{K}_{\delta,\ell})^{-1}\boldsymbol{z}_{n,\ell} -\boldsymbol{f}_\ell 
  =  -\delta(\boldsymbol{K}_{\delta,\ell})^{-1}\boldsymbol{\dot B}_\ell\boldsymbol{f}_\ell+n^{-1/2}(\boldsymbol{K}_{\delta,\ell})^{-1}\boldsymbol{\eta}_\ell.$$
We introduce further the event
$\{\|\delta \boldsymbol{\dot B}_\ell\|_{\text{op}} \leq a_\ell\}$ with 
$a_\ell = \tfrac{\rho}{\kappa_\ell}$ for some $0 < \rho < \tfrac{1}{2}$ 
and the condition $\{\|\boldsymbol{K}_\ell \boldsymbol{f}_\ell\| \geq  \tfrac{\tau_\ell}{2}\}$. We thus have  
$$I \lesssim IV + V + VI+VII,$$
with
\begin{align*}
IV & = \sum_{\ell=1}^L\|\delta(\boldsymbol{K}_{\delta,\ell})^{-1}\boldsymbol{\dot B}_\ell\boldsymbol{f}_\ell\|^2 {\bf 1}_{{\mathcal A}_\ell\, \cap \,{\mathcal B}_\ell}\, {\bf 1}_{\big\{\|\delta \boldsymbol{\dot B}_\ell\|_{\text{op}} \leq a_\ell \big\}}{\bf 1}_{\big\{\|\boldsymbol{K}_\ell \boldsymbol{f}_\ell\| \geq \tfrac{\tau_\ell}{2}\big\}},\\
V & = \sum_{\ell=1}^L\|n^{-1/2}(\boldsymbol{K}_{\delta,\ell})^{-1}\boldsymbol{\eta}_\ell\|^2 {\bf 1}_{{\mathcal A}_\ell\, \cap \,{\mathcal B}_\ell} \,{\bf 1}_{\big\{\|\delta \boldsymbol{\dot B}_\ell\|_{\text{op}} \leq a_\ell \big\}}{\bf 1}_{\big\{\|\boldsymbol{K}_\ell \boldsymbol{f}_\ell\| \geq \tfrac{\tau_\ell}{2}\big\}},\\
VI & = \sum_{\ell=1}^L\|\delta(\boldsymbol{K}_{\delta,\ell})^{-1}\boldsymbol{\dot B}_\ell\boldsymbol{f}_\ell\|^2 {\bf 1}_{{\mathcal A}_\ell\, \cap \,{\mathcal B}_\ell} \Big({\bf 1}_{\big\{\|\delta \boldsymbol{\dot B}_\ell\|_{\text{op}} > a_\ell\big\}}+{\bf 1}_{\big\{\|\boldsymbol{K}_\ell \boldsymbol{f}_\ell\| < \tfrac{\tau_\ell}{2}\big\}}\Big),\\
VII & = \sum_{\ell=1}^L\|n^{-1/2}(\boldsymbol{K}_{\delta,\ell})^{-1}\boldsymbol{\eta}_\ell\|^2 {\bf 1}_{{\mathcal A}_\ell\, \cap \,{\mathcal B}_\ell}\Big({\bf 1}_{\big\{\|\delta \boldsymbol{\dot B}_\ell\|_{\text{op}} > a_\ell\big\}}+{\bf 1}_{\big\{\|\boldsymbol{K}_\ell \boldsymbol{f}_\ell\| < \tfrac{\tau_\ell}{2}\big\}}\Big).
\end{align*}
We shall next successively bound each term $IV$, $V$, $VI$ and $VII$
\end{proof}
\begin{proof}[$\bullet$ The term IV] First, we have
\begin{align*}
(\boldsymbol{K}_{\ell})^{-1} & = (\boldsymbol{K}_{\delta,\ell}-\delta \boldsymbol{\dot B}_\ell)^{-1} \\
& =  (\boldsymbol{I}-\delta \boldsymbol{K}_{\delta,\ell}^{-1}\boldsymbol{\dot B})^{-1}(\boldsymbol{K}_{\delta,\ell})^{-1}.
\end{align*}
On ${\mathcal A}_\ell  = \{\|(\boldsymbol{K}_{\delta,\ell})^{-1}\|_{\text{op}}\leq \kappa_\ell\}$ and $\{\|\delta \boldsymbol{\dot B}_{\ell}\|_{\text{op}} \leq a_\ell\}$, since $a_\ell$ satisfies $\kappa_\ell\, a_\ell = \rho < \tfrac{1}{2}$, by a usual Neumann series argument,
\begin{align*}
\| \big(\boldsymbol{I}-\delta (\boldsymbol{K}_{\delta,\ell})^{-1}\boldsymbol{\dot B}\big)^{-1}\|_{\text{op}} & = \|\sum_{i\geq 0}(-\boldsymbol{K}_{\delta,\ell})^{i}(\delta\boldsymbol{\dot B})^i\|_{\text{op}} \\
& \leq  \sum_{i \geq 0} \|\boldsymbol{K}_{\delta,\ell}\|_{\text{op}}^i\|\delta\boldsymbol{\dot B}\|_{\text{op}}^i\\
& \leq \sum_{i \geq 0} \rho^i=(1-\rho)^{-1}.
\end{align*}
Therefore, on ${\mathcal A}_\ell$ and $\{\|\delta \boldsymbol{\dot B}_{\ell}\|_{\text{op}} \leq a_\ell\}$, we have 
\begin{equation} \label{eq norme bruit}
\|(\boldsymbol{K}_\ell)^{-1}\|_{\text{op}} \leq (1-\rho)^{-1} \|(\boldsymbol{K}_{\delta,\ell})^{-1}\|_{\text{op}} \leq (1-\rho)^{-1}\kappa_\ell.
\end{equation}
Second, we now write
$$(\boldsymbol{K}_{\delta,\ell})^{-1}=\big(\boldsymbol{I}-(\boldsymbol{K}_\ell)^{-1}\delta \boldsymbol{\dot B}_\ell\big)^{-1}(\boldsymbol{K}_\ell)^{-1},$$
hence, on ${\mathcal A}_\ell$ and $\{\|\delta \boldsymbol{\dot B}_{\ell}\|_{\text{op}} \leq a_\ell\}$, we have by \eqref{eq norme bruit}
$$\|(\boldsymbol{K}_{\ell})^{-1}\delta \boldsymbol{\dot B}_{\ell}\|_{\text{op}} \leq (1-\rho)^{-1}\kappa_\ell a_\ell \leq \frac{\rho}{1-\rho}<1$$
since $\rho <\tfrac{1}{2}$ by assumption. The same Neumann series argument now entails
\begin{equation} \label{eq norme bruit inverse}
\|(\boldsymbol{K}_{\delta,\ell})^{-1}\|_{\text{op}} \leq \frac{\rho}{1-\rho}\|(\boldsymbol{K}_\ell)^{-1}\|_{\text{op}}.
\end{equation}
We are ready to bound the term IV itself. We have 
\begin{align*}
&\, \|\delta(\boldsymbol{K}_{\delta,\ell})^{-1}\boldsymbol{\dot B}_\ell\boldsymbol{f}_\ell\|^2 {\bf 1}_{{\mathcal A}_\ell}{\bf 1}_{\big\{\|\delta \boldsymbol{\dot B}_{\ell}\|_{\text{op}} \leq a_\ell\big\}} \\
 \leq &\, \|(\boldsymbol{K}_{\delta,\ell})^{-1}\|_{\text{op}}^2 \|\delta \boldsymbol{\dot B}_{\ell}\|_{\text{op}}^2 \|\boldsymbol{f}_\ell\|^2 \,{\bf 1}_{{\mathcal A}_\ell}{\bf 1}_{\big\{\|\delta \boldsymbol{\dot B}_{\ell}\|_{\text{op}} \leq a_\ell\big\}} \\
 \lesssim &\, \|(\boldsymbol{K}_\ell)^{-1}\|_{\text{op}}^2\kappa_\ell^{-2} \|\boldsymbol{f}_\ell\|^2\,{\bf 1}_{\big\{\|(\boldsymbol{K}_\ell)^{-1}\|_{\text{op}} \leq (1-\rho)^{-1}\kappa_\ell\big\}},
\end{align*}
where we successively used \eqref{eq norme bruit} and \eqref{eq norme bruit inverse}. It follows that
\begin{align*}
\E\big[IV\big] &\lesssim \sum_{\ell=1}^L\|(\boldsymbol{K}_\ell)^{-1}\|_{\text{op}}^2\kappa_\ell^{-2} \|\boldsymbol{f}_\ell\|^2\,{\bf 1}_{\big\{\|(\boldsymbol{K}_\ell)^{-1}\|_{\text{op}} \leq (1-\rho)^{-1}\kappa_\ell\big\}} \\
& \lesssim \sum_{\ell=1}^{L}\ell^{2\nu}\kappa_\ell^{-2}\|\boldsymbol{f}_\ell \|^2
\end{align*}
where we used Assumption \ref{spectral behaviour of K}. The bound is uniform in $K \in {\mathcal G}^{\nu}(Q)$. By definition of $\kappa_\ell$ and using that $|\Lambda_\ell|$ is of order $\ell^{d-1}$, we derive
$$
\E[IV] \lesssim \big((\delta^2|\log\delta|)\bigvee n^{-1}\big) \sum_{\ell=1}^L \ell^{2\nu+d-1}\|\boldsymbol{f}_\ell\|^2.
$$
If $2\nu+d-1 \leq 2s$, we have 
$$\sum_{\ell=1}^L \ell^{2\nu+d-1}\|\boldsymbol{f}_\ell\|^2 \leq \|f\|_{{\mathcal W}^s}^2,$$
therefore 
\begin{align}
\E\big[IV\big] & \lesssim \delta^2|\log \delta| +L^{-2s} \nonumber \\
& \lesssim \big(\delta^2 |\log \delta|\big)^{1 \bigwedge 2s/(2\nu+d-1)} \bigvee n^{-2s/(2\nu+d)} \label{case very regular}
\end{align}
by definition of $L$ in \eqref{choice of L}, and this result is uniform in $f \in {\mathcal W}^s(M)$.
If $2\nu+d-1 \geq 2s$, we have
\begin{align*}
\sum_{\ell=1}^L \ell^{2\nu+d-1}\|\boldsymbol{f}_\ell\|^2 & \leq L^{2(\nu-s)+d-1}\sum_{\ell=1}^L\ell^{2s}\|\boldsymbol{f}_\ell\|^2 \\
&\leq L^{2(\nu-s)+d-1}\|f\|_{{\mathcal W}^s}^2.
\end{align*}
By definition of $L$ again
we derive
\begin{align} \label{case regular}
\E\big[IV\big] \lesssim &\, L^{-2s} L^{2\nu+d-1}(n^{-1}\bigvee \de^2\abs{\log \de}) \nonumber \\
\lesssim &\, L^{-2s}( n^{-1/(2\nu+d)} \bigvee 1) \le L^{-2s}
\end{align}
and this bound is uniform in $f \in {\mathcal W}^s(M)$. Putting together \eqref{case very regular} and \eqref{case regular}, we finally obtain
\begin{equation}  \label{estimation IV}
\E\big[IV\big] \lesssim \big(\delta^2 |\log \delta|\big)^{1 \bigwedge 2s/(2\nu+d-1)} \bigvee n^{-2s/(1\nu+d)}
\end{equation}
uniformly in $f \in {\mathcal W}^s(M), K\in {\mathcal G}^{\nu}(Q)$.
\end{proof}
\begin{proof}[$\bullet$ The term V] We have
\begin{align*}
&\,\|n^{-1/2}(\boldsymbol{K}_{\delta,\ell})^{-1}\boldsymbol{\eta}_\ell\|^2 {\bf 1}_{{\mathcal A}_\ell}{\bf 1}_{\big\{\|\delta \boldsymbol{\dot B}_{\ell}\|_{\text{op}}\leq  a_\ell\big\}}{\bf 1}_{\big\{\|\boldsymbol{K}_\ell \boldsymbol{f}_\ell\| \geq \tfrac{\tau_\ell}{2}\big\}} \\
\leq &\,  n^{-1}\|(\boldsymbol{K}_{\delta,\ell})^{-1}\|_{\text{op}}^2\|\boldsymbol{\eta}_\ell\|^2{\bf 1}_{{\mathcal A}_\ell}{\bf 1}_{\big\{\|\delta \boldsymbol{\dot B}_{\ell}\|_{\text{op}}\leq  a_\ell\big\}}{\bf 1}_{\big\{\|\boldsymbol{K}_\ell \boldsymbol{f}_\ell\| \geq \tfrac{\tau_\ell}{2}\big\}} \\
\lesssim &\, n^{-1}\|(\boldsymbol{K}_{\ell})^{-1}\|_{\text{op}}^2\|\boldsymbol{\eta}_\ell\|^2{\bf 1}_{{\mathcal A}_\ell}{\bf 1}_{\big\{\|\delta \boldsymbol{\dot B}_{\ell}\|_{\text{op}}\leq  a_\ell\big\}}{\bf 1}_{\big\{\|\boldsymbol{K}_\ell \boldsymbol{f}_\ell\| \geq \tfrac{\tau_\ell}{2}\big\}} \\
\lesssim &\,n^{-1}\ell^{2\nu} \|\boldsymbol{\eta}_\ell\|^2 {\bf 1}_{\big\{\|\boldsymbol{K}_\ell \boldsymbol{f}_\ell\| \geq \tfrac{\tau_\ell}{2}\big\}} 
\end{align*}
where we successively used \eqref{eq norme bruit} and \eqref{eq norme bruit inverse} in the same way as for the term IV, the last inequality being obtained thanks to Assumption \ref{spectral behaviour of K}. By Assumption \ref{spectral behaviour of K} again, since 
$$\|\boldsymbol{K}_\ell \boldsymbol{f}_\ell\| \leq \|\boldsymbol{K}_\ell\|_{\text{op}} \|\boldsymbol{f}_\ell\| \leq Q_1(K)\ell^\nu \|\boldsymbol{f}_\ell\|$$
we derive
$${\bf 1}_{\big\{\|\boldsymbol{K}_\ell \boldsymbol{f}_\ell\| \geq \tfrac{\tau_\ell}{2}\big\}} \leq {\bf 1}_{\big\{\|\boldsymbol{f}_\ell\| \geq Q_1(K)^{-1}\tfrac{\tau_\ell}{2} \ell^\nu\big\}}={\bf 1}_{\big\{\|\boldsymbol{f}_\ell\| \geq c\ell^{\nu+(d-1)/2}n^{-1/2}(\log n)^{1/2}\big\}}$$
for some constant $c$ that depends on $Q_1(K)$ and the pre-factor $\mu_0$ in the choice of $\tau_\ell$ only. Since $\E[\|\boldsymbol{\eta}_\ell\|^2] = |\Lambda_\ell| \lesssim \ell^{d-1}$, we infer, for any $1 \leq k \leq L$
\begin{align*}
\E[V] &\, \lesssim n^{-1}\sum_{\ell=1}^L\ell^{2\nu+d-1}{\bf 1}_{\big\{\|\boldsymbol{f}_\ell\| \geq c\,\displaystyle{\ell^{\nu+(d-1)/2}}n^{-1/2}(\log n)^{1/2}\big\}} \\
&\, \lesssim n^{-1}\big(\sum_{\ell=1}^k \ell^{2\nu+d-1}+\sum_{\ell=k+1}^Ln(\log n)^{-1}\|\boldsymbol{f}_\ell\|^2\big) \\
&\, \leq n^{-1}k^{2\nu+d}+(\log n)^{-1}\sum_{\ell > k} \|\boldsymbol{f}_\ell\|^2\\
&\,\lesssim n^{-1}k^{2\nu+d}+(\log n)^{-1}\|f\|_{{\mathcal W}^s}^2k^{-2s}.
\end{align*}
The admissible choice $k = \lfloor \big(n(\log n)^{-1/2}\big)^{1/(2(s+\nu)+d)}\rfloor \land (\de^2)^{-1/(2\nu+d-1)}$ yields
\begin{align} \label{estimation V}
\E[V] & \lesssim  n^{-1}k^{\nu+d}+k^{-2s}  \nonumber \\
& \lesssim \big(n^{-1} \log n\big)^{2s/(2(s+\nu)+d)}+ k^{-2s} \nonumber \\
& \lesssim \big(n^{-1} \log n\big)^{2s/(2(s+\nu)+d)} \bigvee (\delta^2)^{2s/(2\nu+d-1)}
\end{align}
uniformly in $f \in {\mathcal W}^s(M),  K\in {\mathcal G}^{\nu}(Q)$.
\end{proof}
\begin{proof}[$\bullet$ The term VI] We further bound the term $VI$ via
$$VI \leq VIII+IX,$$
with 
\begin{align*}
VIII & = \sum_{\ell=1}^L\|\delta(\boldsymbol{K}_{\delta,\ell})^{-1}\boldsymbol{\dot B}_\ell\boldsymbol{f}_\ell\|^2{\bf 1}_{{\mathcal A}_\ell}{\bf 1}_{\big\{\|\delta \boldsymbol{\dot B}_\ell\|_{\text{op}} > a_\ell\big\}},\\
IX & = \sum_{\ell=1}^L\|\delta(\boldsymbol{K}_{\delta,\ell})^{-1}\boldsymbol{\dot B}_\ell\boldsymbol{f}_\ell\|^2{\bf 1}_{{\mathcal A}_\ell\, \cap \,{\mathcal B}_\ell}{\bf 1}_{\big\{\|\boldsymbol{K}_\ell \boldsymbol{f}_\ell\| < \tfrac{\tau_\ell}{2}\big\}}.
\end{align*}
On ${\mathcal A}_\ell$, we have
$$\|\delta(\boldsymbol{K}_{\delta,\ell})^{-1}\boldsymbol{\dot B}_\ell\boldsymbol{f}_\ell\|^2 \lesssim \delta^2\,\kappa_\ell^2\|\boldsymbol{\dot B}_\ell\|_{\text{op}}^2\|\boldsymbol{f}_\ell\|^2$$
hence
\begin{align*}
\E\big[VIII\big] & \lesssim \delta^2\sum_{\ell=1}^L \kappa_\ell^2\|\boldsymbol{f}_\ell\|^2\E\big[\|\boldsymbol{\dot B}_\ell\|_{\text{op}}^2{\bf 1}_{\big\{\|\delta \boldsymbol{\dot B}_\ell\|_{\text{op}} > a_\ell\big\}}\big] \\
& \leq \delta^2\sum_{\ell=1}^L \kappa_\ell^2\|\boldsymbol{f}_\ell\|^2\E\big[\|\boldsymbol{\dot B}_\ell\|_{\text{op}}^4\big]^{1/2}\PP\big(\|\delta \boldsymbol{\dot B}_\ell\|_{\text{op}} > a_\ell\big)^{1/2} \\
&\lesssim \delta^2\sum_{\ell=1}^L \kappa_\ell^2\|\boldsymbol{f}_\ell\|^2|\Lambda_\ell| \delta^{\,c_0\rho^2|\Lambda_\ell|^2/2\lambda_0^2} \\
& \lesssim |\log \delta| \max_{1 \leq \ell \leq L} \delta^{\,c_0\rho^2|\Lambda_\ell|^2/2\lambda_0^2} \|f\|_{\bH}^2
\end{align*} 
applying successively Cauchy-Schwarz, the moment bound \eqref{moment bound norm operator} and Lemma \ref{concentration operator}. Indeed, since $a_\ell = \rho/\kappa_\ell$, by definition of $\kappa_\ell$ in \eqref{def kappa ell}, we infer
\begin{align} 
\PP\big(\|\delta \boldsymbol{\dot B}_\ell\|_{\text{op}} > a_\ell\big) & \leq \PP\big(|\Lambda_\ell|^{-1/2}\|\boldsymbol{\dot B}_\ell\|_{\text{op}} > |\Lambda_\ell|^{-1/2}\tfrac{\rho}{\kappa_\ell}\delta^{-1}\big) \nonumber \\
& = \PP\big(|\Lambda_\ell|^{-1/2}\|\boldsymbol{\dot B}_\ell\|_{\text{op}} > \tfrac{\rho}{\lambda_0}|\log \delta|^{1/2}\big) \nonumber \\
& \leq \exp\big(-c_0\tfrac{\rho^2}{\lambda_0^2}|\log \delta||\Lambda_\ell|^2\big)=\delta^{\,c_0\rho^2|\Lambda_\ell|^2/\lambda_0^2}  \label{ineg dev op}
\end{align}
by \eqref{first ineq vector} of Lemma \ref{concentration vector gaussian} since $\tfrac{\rho}{\lambda_0}|\log \delta|^{1/2} \geq \beta_0$ for sufficiently small $\lambda_0$ thanks to the assumption $\delta \leq \delta_0 < 1$.
Finally, since $\Lambda_\ell$ is non-empty, by taking $\lambda_0$ sufficiently small, we conclude
\begin{equation} \label{estimation VIII}
\E\big[VIII\big] \lesssim \delta^2
\end{equation}
uniformly in $f \in {\mathcal W}^s(M)$. We now turn to the term $IX$. Observe first that
\begin{equation} \label{trick LD}
{\bf 1}_{{\mathcal B}_\ell} {\bf 1}_{\big\{\|\boldsymbol{K}_\ell \boldsymbol{f}_\ell\| < \tfrac{\tau_\ell}{2}\big\}} \leq {\bf 1}_{\big\{n^{-1/2}\|\boldsymbol{\eta}_\ell\| \geq \tfrac{\tau_\ell}{2}\big\}}.
\end{equation}
We reproduce the steps we used for the term $VIII$, replacing the event $\{\|\delta \boldsymbol{\dot B}_\ell\|_{\text{op}} > a_\ell\}$ by $\{n^{-1/2}\|\boldsymbol{\eta}_\ell\| \geq \tfrac{\tau_\ell}{2}\}$. We obtain
$$\E\big[IX\big] \lesssim \delta^2\sum_{\ell=1}^L \kappa_\ell^2\|\boldsymbol{f}_\ell\|^2|\Lambda_\ell| \PP\big(n^{-1/2}\|\boldsymbol{\eta}_\ell\| \geq \tfrac{\tau_\ell}{2}\big)^{1/2}.$$
By definition of $\tau_\ell$ in \eqref{def tau ell} and Lemma \ref{first ineq vector}, we have
\begin{align}
\PP\big(n^{-1/2}\|\boldsymbol{\eta}_\ell\|>\tfrac{\tau_\ell}{2}\big) & = \PP\big(|\Lambda_\ell|^{-1/2}\|\boldsymbol{\eta}_\ell\| > \tfrac{\mu_0}{2}(\log n)^{1/2}\big) \nonumber\\
& \leq \exp\big(-c_1\tfrac{\mu_0^2}{4}\log n\big) = n^{-c_1\mu_0^2/4} \label{ineg dev norme vector}
\end{align}
since $\tfrac{\mu_0}{2}(\log n)^{1/2} \geq \beta_1$ for large enough $\mu_0$. It follows that
\begin{equation} \label{estimation IX}
\E\big[IX\big] \lesssim |\log \delta|\|f\|_{\bH}^2\,n^{-c_1\mu_0^2/4} \lesssim n^{-1} |\log \delta| 
\end{equation}
by taking $\mu_0$ sufficiently large. The bound is uniform in $f \in {\mathcal W}^s(M)$. Putting together the estimates \eqref{estimation VIII} and \eqref{estimation IX}, we derive
\begin{equation} \label{estimation VI}
\E\big[VI\big] \lesssim \delta^2+n^{-1} |\log \delta| 
\end{equation}
for large enough $n$, uniformly in $f \in {\mathcal W}^s(M)$.
\end{proof}
\begin{proof}[$\bullet$ The term VII]. The arguments needed here are quite similar to those we used for the term $VI$. On ${\mathcal A}_\ell$, we have
$$\|n^{-1/2}(\boldsymbol{K}_{\delta,\ell})^{-1}\boldsymbol{\eta}_\ell\|^2 \leq n^{-1}\kappa_\ell^2\|\boldsymbol{\eta}_\ell\|^2,$$
hence, using \eqref{trick LD}, the fact that $\E\big[\|\boldsymbol{\eta}_\ell\|^2\big] = |\Lambda_\ell| \lesssim \ell^{d-1}$ together with $\kappa_\ell \leq n^{1/2}$ by definition \eqref{def kappa ell}, we successively obtain
\begin{align*}
\E\big[VII\big] & \leq n^{-1}\sum_{\ell= 1}^L\kappa_\ell^2\E\big[\|\boldsymbol{\eta}_\ell\|^2\big]\Big(\PP\big(\|\delta \boldsymbol{\dot B}_\ell\|_{\text{op}} > a_\ell\big)+\PP\big(n^{-1/2}\|\boldsymbol{\eta}_\ell\|>\tfrac{\tau_\ell}{2}\big)\Big) \\
& \lesssim \max_{1 \leq \ell \leq L}\big\{\PP\big(\|\delta \boldsymbol{\dot B}_\ell\|_{\text{op}} > a_\ell\big)+\PP\big(n^{-1/2}\|\boldsymbol{\eta}_\ell\|>\tfrac{\tau_\ell}{2}\big)\big\} \sum_{\ell=1}^L \ell^{d-1}\\
& \lesssim L^{d-1}\big(\delta^{\,c_0\rho^2/\lambda_0^2} + n^{-c_1\mu_0^2/4}\big)
\end{align*}
where we applied \eqref{ineg dev op} and \eqref{ineg dev norme vector} to obtain the last inequality. The choice of $L$ in \eqref{choice of L} leads to
\begin{equation} \label{estimation VII}
\E\big[VII\big] \lesssim (\delta^2)^{-\frac{d-1}{2\nu+d-1}+\frac{c_0\rho^2}{\lambda_0^2}} + n^{\frac{1}{2\nu+d}-\frac{c_1\mu_0^2}{4}} \lesssim \delta^2 \bigvee n^{-1}
\end{equation}
by taking $\lambda_0$ sufficiently small and $\mu_0$ sufficiently large.
\end{proof}
\begin{proof}[$\bullet$ The term I, conclusion] We put together the estimates \eqref{estimation IV}, \eqref{estimation V}, \eqref{estimation VI} and \eqref{estimation VII}. We obtain
\begin{equation} \label{estimation I}
\E\big[I\big] \lesssim \big(\delta^2 |\log \delta|\big)^{1 \bigwedge 2s/(2\nu+d-1)}
\bigvee \big(n^{-1}\log n\big)^{2s/(2(s+\nu)+d)}
\end{equation}
uniformly in $f \in {\mathcal W}^s(M)$.
\end{proof}
\begin{proof}[$\bullet$ The term II]  We claim the following inequality
\begin{equation} \label{inclusion inverse operator}
{\bf 1}_{{\mathcal A}^c} \leq {\bf 1}_{\big\{\|(\boldsymbol{K}_\ell)^{-1}\|_{\text{op}} \geq \tfrac{\kappa_\ell}{2}\big\}}+ {\bf 1}_{\big\{\|\boldsymbol{K}_{\delta,\ell}-\boldsymbol{K}_\ell\|_{\text{op}}\geq \kappa_\ell^{-1}\big\}},
\end{equation}
a consequence of the following elementary lemma
\begin{lemma} \label{elementary lemma}
Let $A$ and $B$ be two bounded operators with bounded inverse. If $\|B^{-1}\| \geq \kappa$ for some $\kappa >0$, then
either $\|A^{-1}\| \geq \kappa/2$ or $\|A-B\| \geq 1/\kappa$.
\end{lemma}
\begin{proof}[Proof of Lemma \ref{elementary lemma}] Write $B=A+\xi$. 
Assume that $\|A^{-1}\|<\kappa/2$. By the triangle inequality, $\|(A+\xi)^{-1}-A^{-1}\| \geq \kappa/2$. We proceed by contradiction: suppose that $\|\xi\| \leq  1/\kappa$. Then we have
$\|A^{-1}\xi\| \leq \|A^{-1}\|\|\xi\| \leq 1/2<1$ and a standard Neumann series argument entails 
\begin{align*}
\|(A+\xi)^{-1}-A^{-1}\| = &  \|(I+A^{-1}\xi)^{-1}A^{-1}-A^{-1}\| \\
= &  \|\sum_{i \geq 1}(-1)^i(A^{-1})^{i+1}\xi^i\| \\
\leq &  \sum_{i \geq 1}\|A^{-1}\|^{i+1} \|\xi\|^i \\
< &  \frac{\kappa}{2} \sum_{i \geq 1}\Big(\frac{\kappa}{2}\Big)^i \Big(\frac{1}{\kappa}\Big)^i = \frac{\kappa}{2},
\end{align*}
a contradiction.
\end{proof}
By Assumption \eqref{spectral behaviour of K}, we have $\|(\boldsymbol{K}_\ell)^{-1}\|_{\text{op}} \leq Q_2(K)\ell^\nu$. Therefore
$${\bf 1}_{\big\{\|(\boldsymbol{K}_\ell)^{-1}\|_{\text{op}} \geq \tfrac{\kappa_\ell}{2}\big\}} \leq {\bf 1}_{\big\{\ell \geq c \big(\delta^2 |\log \delta|\big)^{1/(2\nu+d-1)}\bigwedge n^{1/(2\nu)}\big\}}$$
for some constant $c$ that depends on $Q_2(K)$ and $\lambda_0$ only. For the second term in the right-hand side of \eqref{inclusion inverse operator}, we apply by Lemma \ref{concentration operator}  in the same way as we obtained \eqref{estimation VIII} for the term $VIII$. We derive
\begin{align*}
\PP\big(\|\delta \boldsymbol{\dot B}_\ell\|_{\text{op}} \geq \kappa_\ell^{-1}\big) & = \PP\big(|\Lambda_\ell|^{-1/2}\|\boldsymbol{\dot B}_\ell\|_{\text{op}} \geq \mu_0^{-1}|\log \delta|^{1/2}\big) \\
& \leq \exp\big(-\tfrac{c_0}{\mu_0^2}|\log \delta||\Lambda_n|^2\big)  = \delta^{\,c_0|\Lambda_\ell|^2/\mu_0} \leq  \delta^{\,c_0/\mu_0} 
\end{align*}
for large enough $\mu_0$. Therefore
\begin{align*}
\E\big[II\big]  \leq & \sum_{\ell=1}^L \|\boldsymbol{f}_\ell\|^2\Big({\bf 1}_{\big\{\ell \geq c\big(\delta^2 |\log \delta|\big)^{1/(2\nu+d-1)}\bigwedge n^{1/(2\nu)}\big\}}+\PP\big(\|\delta \boldsymbol{\dot B}_\ell\|_{\text{op}} \geq \kappa_\ell^{-1} \big)\Big) \\
\lesssim &\;\big(n^{-s/\nu} \bigvee \big(\delta^2 |\log \delta|\big)^{2s/(2\nu+d-1)}\big)\|f\|_{{\mathcal W}^s}^2+\|f\|_{\bH}^2\,\delta^{\,c_0/\mu_0}. 
\end{align*}
We finally obtain
\begin{align}
 \E\big[II\big] & \lesssim \big(\delta^2 \log\delta^{-1}\big)^{2s/(2\nu+d-1)} + \delta^2 +n^{-s/\nu} \nonumber \\
 & \lesssim  \big(\delta^2 |\log \delta|\big)^{1 \bigwedge 2s/(2\nu+d-1)} \bigvee \big(n^{-1}\log n\big)^{2s/(2(s+\nu)+d)}  \label{estimation II}
\end{align}
uniformly in $f \in {\mathcal W}^s(M),  K\in {\mathcal G}^{\nu}(Q)$. 
\end{proof}
\begin{proof}[$\bullet$ The term III] Obviously, the decomposition \eqref{matrix gaussian deconvolution} entails
$${\bf 1}_{{\mathcal B}^c}={\bf 1}_{\big\{\|\boldsymbol{K}_\ell \boldsymbol{f}_\ell + n^{-1/2}\boldsymbol{\eta}_\ell\| < \tau_\ell\big\}}\leq  {\bf 1}_{\big\{\|\boldsymbol{K}_\ell \boldsymbol{f}_\ell\| \leq 2 \tau_\ell\big\}}+  {\bf 1}_{\big\{n^{-1/2}\|\boldsymbol{\eta}_\ell\| > \tau_\ell\big\}}.$$
On the one hand, we have
$$\|\boldsymbol{K}_{\ell}\boldsymbol{f}_\ell\| \geq \|(\boldsymbol{K}_\ell)^{-1}\|_{\text{op}}^{-1}\|\boldsymbol{f}_\ell\| \geq Q_2(K)^{-1}{\ell}^{-\nu} \|\boldsymbol{f}_\ell\|$$ 
by Assumption \ref{spectral behaviour of K}. By definition of $\tau_\ell$ in \eqref{def tau ell} it follows that, for any $1 \leq k \leq L$,
\begin{align*}
\sum_{\ell = 1}^L \|\boldsymbol{f}_\ell\|^2{\bf 1}_{\big\{\|\boldsymbol{K}_\ell \boldsymbol{f}_\ell\| \leq 2 \tau_\ell\big\}} & \leq \sum_{\ell=1}^L \|\boldsymbol{f}_\ell\|^2 
{\bf 1}_{\big\{\|\boldsymbol{f}_\ell\| \leq 2Q_2(K)^{-1}\ell^{\nu}\tau_\ell\big\}} \\
& \lesssim \sum_{\ell=1}^k \ell^{2\nu}\tau_\ell^2 + \sum_{\ell=k+1}^{L}\|\boldsymbol{f}_\ell\|^2 \\
& \lesssim (n^{-1}\log n) \sum_{\ell=1}^k \ell^{2\nu+d-1}+\|f\|_{{\mathcal W}^s}^2k^{-2s} \\
& \lesssim (n^{-1} \log n)\, k^{2\nu+d}+\|f\|_{{\mathcal W}^s}^2k^{-2s}.
\end{align*}
The choice $k = \lfloor \big(n^{1/2}(\log n)^{-1/2}\big)^{1/(2(s+\nu)+d)}\rfloor$ yields
\begin{equation} \label{est prelim III}
\sum_{\ell = 1}^L \|\boldsymbol{f}_\ell \|^2{\bf 1}_{\big\{\|\boldsymbol{K}_\ell \boldsymbol{f}_\ell\| \leq 2 \tau_\ell\big\}} \lesssim \big(n^{-1}\log n\big)^{2s/(2(s+\nu)+d)}
\end{equation} 
uniformly in $f \in {\mathcal W}^s(M),  K\in {\mathcal G}^{\nu}(Q)$. On the other hand, by \eqref{ineg dev norme vector}, we have
$$\sum_{\ell=1}^L\|\boldsymbol{f}_\ell \|^2\PP\big(n^{-1/2}\|\boldsymbol{\eta}_\ell\| > \tau_\ell\big\}\big) \lesssim  \|f\|_{\bH}^2\,n^{-c_1\mu_0^2/4} \lesssim n^{-1}$$
by taking $\mu_0$ large enough, uniformly in $f \in {\mathcal W}^s(M)$. Combining this last estimate with \eqref{est prelim III} we infer
\begin{equation} \label{estimation III}
\E\big[III\big] \lesssim \big(n^{-1}\log n\big)^{2s/(2(s+\nu)+d)}+n^{-1}
\end{equation}
uniformly in $f \in {\mathcal W}^{s}(M),  K\in {\mathcal G}^{\nu}(Q)$.
\end{proof}
\begin{proof}[Proof of Theorem \ref{upper bounds}, completion]
It remains to piece together the estimates \eqref{controle biais}, \eqref{estimation I}, \eqref{estimation II} and \eqref{estimation III}.
\end{proof} 
\subsection{Proof of Theorem \ref{lower bounds}} 
\subsubsection*{Preliminaries: a Bayesian inequality}
For every $\ell \geq 1$, denote by ${\mathcal M}_\ell$ the set of $|\Lambda_\ell|\times|\Lambda_\ell|$ matrices. We denote by ${\mathcal M}^\nu_\ell(Q)$ the subset of ${\mathcal M}_\ell$ of matrices $\boldsymbol{K}_\ell$ such that
$$\|\boldsymbol{K}_\ell\|_{\mathrm{op}}\leq Q_2\ell^{-\nu}\;\;\text{and}\;\;\|(\boldsymbol{K}_\ell)^{-1}\|_{\mathrm{op}}\leq Q_1\ell^{\nu}.$$
Define
\begin{equation} \label{def K0}
{\boldsymbol K}^0_\ell = c_1\ell^{-\nu} {\boldsymbol I}_\ell
\end{equation}
where ${\boldsymbol I}_\ell$ denotes the identity in ${\mathcal M}_\ell$ and $c_1>0$ is such that
$$ 1/Q_1 < c_1 < Q_2$$
so that $\boldsymbol{K}_\ell^0 \in {\mathcal M}_\ell^\nu(Q)$.
We assume a Bayesian approach and pick $\boldsymbol{K}_\ell$ at random, with 
$$\boldsymbol{K}_\ell = \boldsymbol{K}_\ell^0+c_2\,\delta \boldsymbol{\dot W}_\ell,$$
for some $c_2>0$ and where $\boldsymbol{\dot W}_\ell$ is an independent copy of $\boldsymbol{\dot B}_\ell$. Define $\boldsymbol{g}_\ell = (1\;0\;\ldots\; 0)^T$ as the first canonical (column) vector in $\R^{|\Lambda_\ell|}$. Define also
\begin{equation} \label{def theta}
\boldsymbol{\vartheta} = -(\boldsymbol{K}_\ell^0)^{-1}(\boldsymbol{K}_\ell-\boldsymbol{K}_\ell^0)(\boldsymbol{K}_\ell^0)^{-1}\boldsymbol{g}_\ell
\end{equation}
and
\begin{equation} \label{def of X}
\boldsymbol{X}=-(\boldsymbol{K}_\ell^0)^{-1}(\boldsymbol{K}_{\delta,\ell}-\boldsymbol{K}_\ell^0)(\boldsymbol{K}_\ell^0)^{-1}\boldsymbol{g}_\ell.
\end{equation}
\begin{lemma}\label{bayesian inequality} There exists a constant $c_3$ depending on $\nu,Q$ and $c_2$ only such that
\begin{equation} \label{inegalite bayes}
\inf_T\PP\big(\delta^{-2}\ell^{-4\nu}|\Lambda_\ell|^{-1}\|T(\boldsymbol{X})-\boldsymbol{\vartheta}\|^2\geq c_3\big) \geq \tfrac{1}{2},
\end{equation}
where the infimum is taken among all estimators $T$ based on the observation $\boldsymbol{X}$.
\end{lemma}
\begin{proof}[Proof of Lemma \ref{bayesian inequality}]
We have $\boldsymbol{X}=\boldsymbol{\vartheta} + \boldsymbol{\varepsilon}$, with
$$\boldsymbol{\vartheta} =  -(\boldsymbol{K}_\ell^0)^{-1}c_2\delta \boldsymbol{\dot W}(\boldsymbol{K}_\ell^0)^{-1}\boldsymbol{g}_\ell\;\;\text{and}\;\;
\boldsymbol{\varepsilon} = -(\boldsymbol{K}_\ell^0)^{-1}\delta \boldsymbol{\dot B}(\boldsymbol{K}_\ell^0)^{-1}\boldsymbol{g}_\ell.$$
By construction, $\boldsymbol{\vartheta}$ and $\boldsymbol{\varepsilon}$ are two independent Gaussian random vectors. More precisely, by definition of $\boldsymbol{g}_\ell$ and with obvious notation, we have
$$\boldsymbol{\vartheta} \sim {\mathcal N}\big(0,\delta^2c_2^2 c_1^{-4}\ell^{4\nu}\boldsymbol{I}_\ell\big)\;\;\text{and}\;\;\boldsymbol{\varepsilon} \sim {\mathcal N}\big(0,\delta^2c_1^{-4}\ell^{4\nu}\boldsymbol{I}_\ell\big).$$
It readily follows that the posterior law of $\boldsymbol{\vartheta}$ given $\boldsymbol{X}$ is
$${\mathcal L}(\boldsymbol{\vartheta}\,\big|\,\boldsymbol{X}) = {\mathcal N}\Big(\frac{c_2^2}{1+c_2^2}\boldsymbol{X}, \delta^2\frac{c_2^2}{1+c_2^2}c_1^{-4}\ell^{4\nu}\boldsymbol{I}_\ell\Big).$$
Now, for $c_3>0$, define 
$$H_\delta(c_3,\boldsymbol{x}) = {\bf 1}_{\displaystyle \{\delta^{-2}\ell^{-4\nu}|\Lambda_\ell|^{-1}\|\boldsymbol{x}\|^2 \geq c_3\}}\;\;\text{for}\;\;\boldsymbol{x}\in \R^{|\Lambda_\ell|}.$$ 
Setting $z(\boldsymbol{X})=T(\boldsymbol{X})-\E[\boldsymbol{\vartheta}\,|\,\boldsymbol{X}]$, we have
\begin{align*}
 \E\big[H_\delta\big(c_3,T(\boldsymbol{X})-\boldsymbol{\vartheta}\big)\,|\,\boldsymbol{X}\big] & = \E\big[H_\delta\big(c_3,z(\boldsymbol{X})+\E[\boldsymbol{\vartheta}\,|\,\boldsymbol{X}]-\boldsymbol{\vartheta}\big)\,|\,\boldsymbol{X}\big]\\
& \geq  \E\big[H_\delta\big(c_3,\E[\boldsymbol{\vartheta}\,|\,\boldsymbol{X}]-\boldsymbol{\vartheta}\big)\,|\,\boldsymbol{X}\big]
\end{align*}
where we used a version of Anderson's Lemma given in Lemma 10.2 in \cite{IH} p. 157. Indeed, the law of $\E[\boldsymbol{\vartheta}\,|\,\boldsymbol{X}]-\vartheta$ has a centrally symmetric density and the function $H_\delta$ is nonnegative, centrally symmetric, satisfies $H_\delta(0)=0$ and the sets $\{\boldsymbol{x},\;H_\delta(c_3,\boldsymbol{x}) < c\}$ are convex for any $c>0$.

Now, $\|\E[\boldsymbol{\vartheta}\,|\,\boldsymbol{X}]-\boldsymbol{\vartheta}\|^2$ has a $\chi^2$-distribution with $|\Lambda_\ell|$ degrees of freedom, up to a scaling factor of order $\delta^2\ell^{4\nu}$. This means that the sequence of random variables $\delta^{-2}\ell^{-4\nu}|\Lambda_\ell|^{-1}\|\E[\boldsymbol{\vartheta}\,|\,\boldsymbol{X}]-\boldsymbol{\vartheta}\|^2$ is bounded below in probability in $\ell \geq 1$ and $\delta>0$. Since $\E[\boldsymbol{\vartheta}\,|\,\boldsymbol{X}]-\boldsymbol{\vartheta}$ is moreover independent of $\boldsymbol{X}$, it follows that there exists $c_3$ independent of $\delta$ and  $\ell$ such that 
$$\E\big[H_\delta\big(c_3, \E[\boldsymbol{\vartheta}\,|\,\boldsymbol{X}]-\boldsymbol{\vartheta}\big)\,|\,\boldsymbol{X}\big] \geq \tfrac{1}{2}.$$
Integrating with respect to $\boldsymbol{X}$, we obtain \eqref{inegalite bayes} and the result follows.
\end{proof}
\subsubsection*{Proof of Theorem \ref{lower bounds}}
We assume with no loss of generality that $2\nu +d -1 \geq  2s$. (Otherwise, the lower bound $\delta$ trivially follows from the parametric case.)
Let $\Pi^{s,\nu}(M,Q_1)$ denote the set of sequences $\pi = (\pi_\ell)_{\ell \geq 1}$ satisfying
\begin{equation} \label{contrainte piell}
\sum_{\ell \geq 1} \pi_\ell^2 \ell^{2(s+\nu)} \leq \frac{M^2}{Q_1^2}.
\end{equation}
For $\pi \in \Pi^{s,\nu}(M,Q_1)$ and $K \in {\mathcal G}^\nu(Q)$, define $f$ via its coordinates in $H_\ell$ by
$${\boldsymbol f}_\ell = \pi_\ell {\boldsymbol K}_\ell^{-1} {\boldsymbol g}_\ell,\;\;\ell \geq 1,$$
where ${\boldsymbol g}_\ell$ is an arbitrary vector in $\R^{|\Lambda_\ell|}$ with $\|{\boldsymbol g}_\ell\|=1$ (fixed in the sequel). Then
$$\sum_{\ell \geq 1}\ell^{2s}\|\pi_\ell {\boldsymbol K}_\ell^{-1}{\boldsymbol g}_\ell\|^2 \leq \sum_{\ell \geq 1} \pi_\ell^2 \|{\boldsymbol K}_\ell^{-1}\|_{\text{op}}^2\|{\boldsymbol g}_\ell\|^2 \leq Q_1^2 \sum_{\ell \geq 1}\pi_\ell^2 \ell^{2(s+\nu)} \leq M^2$$
since $\pi \in \Pi^{s,\nu}(M,Q_1)$. Therefore $f \in {\mathcal W}^s(M)$. It follows that for an arbitrary estimator $\widehat f$, we have
\begin{align*}
 & \sup_{f\in {\mathcal W}^{s}(M),  K\in {\mathcal G}^\nu(Q)}
  \E\Big[\big\|\widehat f-f\big\|_{\bH}^2\Big] \\
  = & \sup_{f\in {\mathcal W}^{s}(M),  K\in {\mathcal G}^\nu(Q)}
 \sum_{\ell \geq 1}\E\Big[\big\|\widehat {\boldsymbol f}_\ell-{\boldsymbol f}_\ell\big\|^2\Big] \\
 \geq & \sup_{\pi \in \Pi^{s,\nu}(M,Q_1), K \in {\mathcal G}^\nu(Q)}\sum_{\ell \geq 1}\E\Big[\big\|\widehat {\boldsymbol f}_\ell-\pi_\ell {\boldsymbol K}^{-1}_\ell {\boldsymbol g}_\ell\big\|^2\Big].
\end{align*}
\begin{lemma} \label{Bayesian LB}
There exist a choice of $\boldsymbol{g}_\ell$ with $\|\boldsymbol{g}_\ell\|=1$ and constants $c_4, c_5$ (depending on $s,\nu, M, Q$) such that for any $\pi \in \Pi^{s,\nu}(M,Q_1)$, if $|\Lambda_\ell|^{1/2}\delta \leq c_4\,\ell^{-\nu}$, we have
\begin{equation} \label{block LB}
\inf_{\widehat {\boldsymbol f}_\ell}\sup_{K \in {\mathcal G}^{\nu}(Q)} 
\E\big[\|\widehat {\boldsymbol f}_\ell -\pi_\ell{\boldsymbol K}_\ell^{-1} {\boldsymbol g}_\ell\|^2\big] \geq c_5\,\delta^2\ell^{4\nu + d-1}\pi_\ell^2
\end{equation}
where the infimum is taken over all estimators and provided $\delta >0$ is sufficiently small. 
\end{lemma}
With \eqref{block LB}, we easily conclude: Define  $L=\lfloor c_6\delta^{-2/(2\nu+d-1)}\rfloor$ with $c_6>0$. For $1 \leq \ell \leq L$, the assumption $|\Lambda|^{1/2}\delta \leq c_4\, \ell^{-\nu}$ of Lemma \ref{Bayesian LB} is satisfied by picking $c_6>0$ sufficiently small and we have
\begin{align*}
& \sup_{\pi \in \Pi^{s,\nu}(M,Q_1), K \in {\mathcal G}^\nu(Q)}\sum_{\ell \geq 1}\E\Big[\big\|\widehat {\boldsymbol f}_\ell-\pi_\ell {\boldsymbol K}^{-1}_\ell {\boldsymbol g}_\ell\big\|^2\Big] \\
\geq &\,c_5\delta^2 \sup_{\pi \in \Pi^{s,\nu}(M,Q_1)} \sum_{\ell=1}^L\ell^{4\nu + d-1}\pi_\ell^2 \\
\geq &\,c_5\delta^2  \tfrac{M^2}{Q_1^2}L^{2\nu+d-1-2s}
\geq \,c_5c_6^{2\nu+d-1-2s}\tfrac{M^2}{Q_1^2} \delta^{2s/(2\nu+d-1)}
\end{align*}
thanks to the admissible choice $\pi$ specified by $\pi_\ell^2=\ell^{-2(\nu+s)}M^2/Q_1^2$ if $\ell=L$ and $0$ otherwise. Theorem \ref{lower bounds} follows. It remains to prove Lemma \ref{Bayesian LB}.
\begin{proof}[Proof of Lemma \ref{Bayesian LB}] 
In view of \eqref{block LB}, we may (and will) assume that $\pi_\ell=1$. We rely on the notation and definition of the preliminaries. Observe first that  
\begin{align}
& \inf_{\widehat {\boldsymbol f}_\ell}\sup_{K \in {\mathcal G}^{\nu}(Q)} 
\E\big[\|\widehat {\boldsymbol f}_\ell -{\boldsymbol K}_\ell^{-1} {\boldsymbol g}_\ell\|^2\big] \nonumber\\
= & \inf_{\widehat {\boldsymbol f}_\ell}\sup_{K \in {\mathcal G}^{\nu}(Q)} 
\E\big[\|\widehat {\boldsymbol f}_\ell -\big({\boldsymbol K}_\ell^{-1}- ({\boldsymbol K}^0_\ell)^{-1}\big){\boldsymbol g}_\ell\|^2\big]. \nonumber
\end{align}
where $\boldsymbol{K}^0$ is defined in \eqref{def K0}. Put $v_{\delta,\ell}=\delta^2\ell^{4\nu + d-1}$. For any $c>0$, by Chebyshev inequality, we have 
\begin{align}
& c^2v_{\delta,\ell}^{-2} \inf_{\widehat {\boldsymbol f}_\ell}\sup_{K \in {\mathcal G}^{\nu}(Q)} \E\big[\|\widehat {\boldsymbol f}_\ell -\big({\boldsymbol K}_\ell^{-1}- ({\boldsymbol K}^0_\ell)^{-1}\big){\boldsymbol g}_\ell\|^2\big] \nonumber \\
\geq &  \inf_{\widehat {\boldsymbol f}_\ell}\sup_{K \in {\mathcal G}^{\nu}(Q)} \PP\big(\|\widehat {\boldsymbol f}_\ell -\big({\boldsymbol K}_\ell^{-1}- ({\boldsymbol K}^0_\ell)^{-1}\big){\boldsymbol g}_\ell\|\geq c\,v_{\delta,\ell}\big). \label{binf translatee}
\end{align}
We adopt the same Bayesian approach as in the preliminaries and consider ${\boldsymbol K}_\ell$ as a random matrix with distribution such that
\begin{equation} \label{def prior}
{\boldsymbol K}_\ell = {\boldsymbol K}_\ell^0+c_2\,\delta \boldsymbol{\dot  W}_\ell,
\end{equation}
where $\boldsymbol{\dot W}_\ell$ is an independent copy of $\boldsymbol{\dot B}_\ell$ and $c_2>0$ is to be specified later. 
Using the randomisation \eqref{def prior} on $\boldsymbol{K}_\ell$, the right-hand side in \eqref{binf translatee} is now bigger
than
\begin{equation} \label{mesure a priori}
 \inf_{\widehat {\boldsymbol f}_\ell}\PP\big(\|\widehat {\boldsymbol f}_\ell -\big({\boldsymbol K}_\ell^{-1}- ({\boldsymbol K}^0_\ell)^{-1}\big){\boldsymbol g}_\ell\|\geq c\,v_{\delta,\ell}\big)- \PP\big(\boldsymbol{K}_\ell \notin {\mathcal M}_{\ell}^\nu(Q)\big).
\end{equation}
Let us first show that
\begin{equation} \label{first lb}
  \inf_{\widehat {\boldsymbol f}_\ell}\PP\big(\|\widehat {\boldsymbol f}_\ell -\big({\boldsymbol K}_\ell^{-1}- ({\boldsymbol K}^0_\ell)^{-1}\big){\boldsymbol g}_\ell\|\geq c\,v_{\delta,\ell}\big)
\end{equation}
is bounded below for an appropriate choice of $c>0$. 
Introduce the event
$${\mathcal A}_\delta = \big\{Q_1\ell^\nu c_2\delta\|\boldsymbol{\dot W}_\ell\|_{\mathrm{op}} \leq \rho\big\}$$
for some $0 < \rho < 1$. Observe that $\|(\boldsymbol{K}^0_\ell)^{-1}c_2\delta\boldsymbol{\dot W}_\ell\|_{\mathrm{op}} \leq \rho$ on ${\mathcal A}_\delta$, therefore, by an usual Neuman series argument, we have the decomposition
\begin{align*}
& \boldsymbol{K}_\ell^{-1}-(\boldsymbol{K}^0_\ell)^{-1} \\
 = & -(\boldsymbol{K}_\ell^0)^{-1}(c_2\delta \boldsymbol{\dot W}_\ell)(\boldsymbol{K}^0_\ell)^{-1}+\sum_{n \geq 2}(-1)^n\big((\boldsymbol{K}_\ell^0)^{-1}c_2\boldsymbol{\dot W}_\ell\big)^n(\boldsymbol{K}_\ell^0)^{-1}
\end{align*}
Applying the vector $\boldsymbol{g}_\ell = (1,0,\ldots, 0)$ and setting
$$\boldsymbol{\zeta}_{\delta,\ell} = \sum_{n \geq 2}(-1)^n\big((\boldsymbol{K}_\ell^0)^{-1}c_2\boldsymbol{\dot W}_\ell\big)^n(\boldsymbol{K}_\ell^0)^{-1}{\boldsymbol g}_\ell,$$
we obtain the decomposition
\begin{align*}
 \big(\boldsymbol{K}_\ell^{-1}-(\boldsymbol{K}^0_\ell)^{-1} \big)\boldsymbol{g}_\ell =  & -(\boldsymbol{K}_\ell^0)^{-1}(c_2\delta \boldsymbol{\dot W}_\ell)(\boldsymbol{K}^0_\ell)^{-1}\boldsymbol{g}_\ell +\boldsymbol{\zeta}_{\delta,\ell} \\
 = &\; \boldsymbol{\vartheta} + \boldsymbol{\zeta}_{\delta,\ell},
\end{align*}
where $\boldsymbol{\vartheta}$ is defined in \eqref{def theta}. 
We derive, for any $c>0$
\begin{align*}
& \PP\big(\|\widehat {\boldsymbol f}_\ell -\big({\boldsymbol K}_\ell^{-1}- ({\boldsymbol K}^0_\ell)^{-1}\big){\boldsymbol g}_\ell\|\geq c\,v_{\delta,\ell}\big)\\
\geq &\, \PP\big(\|\widehat {\boldsymbol f}_\ell - (\boldsymbol{\vartheta} +\boldsymbol{\zeta}_{\delta,\ell})\| \geq c\,v_{\delta,\ell}\;\;\text{and}\;\;{\mathcal A}_{\delta}\big)\\
\geq &\, \PP\big(\|\widehat {\boldsymbol f}_\ell - \boldsymbol{\vartheta} \| \geq \tfrac{1}{2}c\,v_{\delta,\ell}\;\;\text{and}\;\;{\mathcal A}_{\delta}\;\;\text{and}\;\; \|\boldsymbol{\zeta}_{\delta,\ell}\| \leq\, \tfrac{1}{2}c\, v_{\delta,\ell}\big) 
\end{align*}
by the triangle inequality. We claim that for any $\varepsilon >0$, there exists a choice of sufficiently small $c_2$ such that for any $c>0$:
\begin{equation} \label{le reste}
\limsup_{\de\rightarrow 0} \PP\big({\mathcal A}_{\delta}\;\;\text{and}\;\;\|\boldsymbol{\zeta}_{\delta,\ell}\| \leq \tfrac{1}{2}c\, v_{\delta,\ell}\big) \geq 1-\varepsilon.
\end{equation}
Let us admit temporarily \eqref{le reste}. For such a choice, we thus have
\begin{align*}
& \PP\big(\|\widehat {\boldsymbol f}_\ell -\big({\boldsymbol K}_\ell^{-1}- ({\boldsymbol K}^0_\ell)^{-1}\big){\boldsymbol g}_\ell\|\geq c\,v_{\delta,\ell}\big)\\
\geq & \PP\big(\|\widehat {\boldsymbol f}_\ell - \boldsymbol{\vartheta} \| \geq \tfrac{1}{2}c\,v_{\delta,\ell}\big)-\varepsilon.
\end{align*}
Let us now look at an apparently different problem: we want to estimate $\boldsymbol{\vartheta}$ from our observation $\boldsymbol{K}_{\delta,\ell}$, or equivalently, from the observation
$$-(\boldsymbol{K}_\ell^0)^{-1}(\boldsymbol{K}_{\delta,\ell}-\boldsymbol{K}_\ell^0)(\boldsymbol{K}_\ell^0)^{-1}.$$
The choice $\boldsymbol{g}_\ell=(1,0,\ldots, 0)^T$ entails that $-(\boldsymbol{K}_\ell^0)^{-1}(\boldsymbol{K}_{\delta,\ell}-\boldsymbol{K}_\ell^0)(\boldsymbol{K}_\ell^0)^{-1}\boldsymbol{g}_\ell$ is a sufficient statistic, but this last quantity is precisely $\boldsymbol{X}$ defined in \eqref{def of X}. Thus, without loss of generality, $\widehat {\boldsymbol f}_\delta$ can be taken as an estimator of the form $T(\boldsymbol{X})$. By Lemma \ref{bayesian inequality}, we know that 
$v_{\delta,\ell}$ is a lower bound for estimating $\boldsymbol{\vartheta}$. 

More specifically, by taking $c$ such that $c\le 2\sqrt{c_3}$, we have 
$$\PP\big(\|\widehat {\boldsymbol f}_\ell - \boldsymbol{\vartheta} \| \geq \tfrac{1}{2}c\,v_{\delta,\ell}\big) -\varepsilon \geq \tfrac{1}{2}-\varepsilon \geq \tfrac{1}{4}$$
say, since the choice of $\varepsilon$ is arbitrary, and \eqref{first lb} follows. It remains to prove \eqref{le reste}. 
 
First, we have that $|\Lambda_\ell|^{-1/2}\|\boldsymbol{\dot W}_\ell\|_{\mathrm{op}}$ is bounded in probability by Lemma \ref{concentration operator} in $\ell \geq 1$. Since $|\Lambda_\ell|^{1/2}\delta \leq c_4\ell^\nu$ by assumption, we also have that $\ell^\nu\delta \|\boldsymbol{\dot W}_\ell\|_{\mathrm{op}}$ is bounded in probability, hence the probability of ${\mathcal A}_\delta$ can be taken arbitrarily close to $1$ by taking $c_2$ sufficiently small. Moreover, on ${\mathcal A}_\delta$, we have
\begin{align*}
\|\zeta_{\delta, \ell}\|  \leq &\,Q_1\ell^\nu \sum_{n \geq 2} \big(Q_1\ell^\nu c_2\delta \|\boldsymbol{\dot W}_\ell\|_{\mathrm{op}}\big)^n 
\\
 \leq &\, (1-\rho)^{-1}   c_2^2\,Q_1^3 \delta^2\ell^{3\nu}   \|\boldsymbol{\dot W}_\ell\|_{\mathrm{op}}^2 \\
 \leq &\, (1-\rho)^{-1}  c_2^2\,Q_1^3 \delta\ell^{2\nu} |\Lambda_\ell|^{1/2}c_4 |\Lambda_\ell|^{-1}\|\boldsymbol{\dot W}_\ell\|_{\mathrm{op}}^2
  \end{align*}
where we again used the fact that $|\Lambda_\ell|^{1/2}\delta \leq c_4\ell^{-\nu}$ by assumption. The claim follows from the fact that $|\Lambda_\ell|^{-1/2}\|\boldsymbol{\dot W}_\ell\|_{\mathrm{op}}$ is bounded in probability. Hence \eqref{le reste} and \eqref{first lb} is proved.

In order to complete the proof of Lemma \ref{Bayesian LB}, we need to check that the term $\PP\big(\boldsymbol{K}_\ell \notin {\mathcal M}_{\ell}^\nu(Q)\big)$ can be taken arbitrarily small when bounding \eqref{binf translatee} below by \eqref{mesure a priori}. We have
\begin{align} 
& \PP\big(\boldsymbol{K}_\ell \notin {\mathcal M}_{\ell}^\nu(Q)\big)  \nonumber \\
\leq & \PP\big(\|\boldsymbol{K}_\ell\|_{\mathrm{op}} > Q_2\ell^{-\nu}\big)
+ \PP\big(\|\boldsymbol{K}^{-1}_\ell\|_{\mathrm{op}} > Q_1\ell^{\nu}\big). \label{controle contrainte}
\end{align}
For the first term in the right-hand side of \eqref{controle contrainte}, we have
\begin{align*}
\PP\big(\|\boldsymbol{K}_\ell\|_{\mathrm{op}} > Q_2\ell^{-\nu}\big)
& \leq \PP\big(\|c_2\delta \boldsymbol{\dot W}_\ell\|_{\mathrm{op}} >Q_2\ell^{-\nu}-\|\boldsymbol{K}_\ell^0\|_{\mathrm{op}}\big) \\
& \leq \PP\big(\|c_2\delta \boldsymbol{\dot W}_\ell\|_{\mathrm{op}} >(Q_2-c_1)\ell^{-\nu}\big). 
\end{align*}
The last term can be rewritten as
$$\PP\big(|\Lambda_\ell|^{-1/2}\|\boldsymbol{\dot W}_\ell\|_{\mathrm{op}} > (Q_2-c_1)c_2^{-1}\ell^{-\nu}|\Lambda_\ell|^{-1/2}\delta^{-1}\big).$$
For the second term in the right-hand side of \eqref{controle contrainte}, thanks to the property $\|\boldsymbol{K}^{-1}_\ell\|_{\mathrm{op}} \le \bigl(c_1\ell^{-\nu}-
\|c_2\de\boldsymbol{\dot W}_\ell\|_{\mathrm{op}}\bigr)^{-1}$ we derive
\begin{align*}
& \PP\big(\|\boldsymbol{K}^{-1}_\ell\|_{\mathrm{op}} > Q_1\ell^{\nu}\big) \\
\le &
\PP\big(|\Lambda_\ell|^{-1/2}\|\boldsymbol{\dot W}_\ell\|_{\mathrm{op}} >(c_1-Q_1^{-1})
c_2^{-1}\ell^{-\nu}|\Lambda_\ell|^{-1/2}\delta^{-1}\big).
\end{align*}
By assumption, we have that $\ell^{-\nu}|\Lambda_\ell|^{-1/2}\delta^{-1}$ is bounded away from zero. Since  $|\Lambda_\ell|^{-1/2}\|\boldsymbol{\dot W}_\ell\|_{\mathrm{op}}$ is tight in $\ell \geq 1$, we can conclude by taking $c_2$ sufficiently small.
The proof of Lemma \ref{Bayesian LB} is complete.

\end{proof}
\newpage

{\small
\subsubsection*{Acknowledgements}
The research of M. Hoffmann is partly supported by the French Agence Nationale de la Recherche (Blanc SIMI 1 2011 project CALIBRATION).
The research of D. Picard is partly supported by the French Agence Nationale de la Recherche (ANR-09-BLAN-0128 PARCIMONIE). We are grateful to J. Rousseau for helpful comments. 
}

\bibliographystyle{plain}       
\bibliography{BiblioDHPV}           

\end{document}